\documentclass[11pt]{amsart}
\usepackage{amscd}
\usepackage{amssymb}
\usepackage{hyperref}
\usepackage[dvips]{graphics}
\usepackage{epsfig,epsf}
\usepackage{mathrsfs}
\usepackage{paralist}
\usepackage{verbatim}

\input{xy}
\xyoption{all}

\newcommand\barmu{{\bar\mu}}

\newcommand\CC{\mathbb{C}}

\newcommand\NN{\mathbb{N}}

\newcommand\PP{\mathbb{P}}

\newcommand\RR{\mathbb{R}}

\newcommand\ZZ{\mathbb{Z}}

\newcommand\bga{{\boldsymbol{\gamma}}}

\newcommand\bchi{{\boldsymbol{\chi}}}

\newcommand\bc{{\mathbf{c}}}

\newcommand\bL{{\mathbf{L}}}

\newcommand\bx{{\mathbf{x}}}

\newcommand\CSO{\operatorname{CSO}}

\newcommand{\rd}{\partial}

\newcommand\thalf{{\textstyle{\frac{1}{2}}}}

\newcommand\half{{{\frac{1}{2}}}}

\newcommand\fs{{\mathfrak{s}}}
\newcommand\fS{{\mathfrak{S}}}
\newcommand\ft{{\mathfrak{t}}}

\newcommand\eps{\varepsilon}

\newcommand\La{\Lambda}
\newcommand\ka{\kappa}

\newcommand\Om{\Omega}
\newcommand\si{\sigma}
\newcommand\Si{\Sigma}

\newcommand\su{{\mathfrak{s}\mathfrak{u}}}

\newcommand\SO{\operatorname{SO}}

\newcommand\less{\setminus}

\newcommand\asd{{\operatorname{asd}}}

\newcommand\Hom{\operatorname{Hom}}

\newcommand\Ker{\operatorname{Ker}}

\newcommand\PD{\operatorname{PD}}

\newcommand\red{\operatorname{red}}
\newcommand\Red{\operatorname{Red}}

\newcommand\SW{SW}
\newcommand\Sym{\operatorname{Sym}}

\newcommand\id{{\mathrm{id}}}

\newcommand\spinc{\text{$\text{spin}^c$ }}
\newcommand\spinu{\text{$\text{spin}^u$ }}
\newcommand\Spinc{\text{$\text{Spin}^c$}}
\newcommand\Spinu{\text{$\text{Spin}^u$}}
\newcommand\vir{\text{vir}}

\newcommand\sM{{\mathscr{M}}}

\newcommand\sO{{\mathcal{O}}}

\newcommand\sU{{\mathcal{U}}}
\newcommand\sV{{\mathcal{V}}}
\newcommand\sW{{\mathcal{W}}}

\newtheorem{thm}{Theorem}[section]
\newtheorem{lem}[thm]{Lemma}

\newtheorem{cor}[thm]{Corollary}

\newtheorem{prop}[thm]{Proposition}

\theoremstyle{definition}

\renewcommand{\thecase}{}

\newtheorem{conj}[thm]{Conjecture}

\newtheorem{hyp}[thm]{Hypothesis}

\newtheorem{rmk}[thm]{Remark}
\renewcommand{\thestep}{}

\theoremstyle{remark}

\makeatletter
\def\alphenumi{
  \def\theenumi{\alph{enumi}}
  \def\p@enumi{\theenumi}
  \def\labelenumi{(\@alph\c@enumi)}}
\makeatother

\makeatletter
\def\thecase{\@arabic\c@case}
\makeatother

\numberwithin{equation}{section}

\makeatletter
\def\thestep{\@arabic\c@step}
\makeatother

\begin{document}
\title[Monopole cobordism and superconformal simple type]
{The SO(3) monopole cobordism and superconformal simple type}
\author[Paul M. N. Feehan]{Paul M. N. Feehan}
\address{Department of Mathematics\\
Rutgers, The State University of New Jersey\\
Piscataway, NJ 08854-8019\\
United States of America}
\email{feehan@math.rutgers.edu}
\urladdr{math.rutgers.edu/$\sim$feehan}
\author[Thomas G. Leness]{Thomas G. Leness}
\address{Department of Mathematics\\
Florida International University\\
Miami, FL 33199\\
United States of America}
\email{lenesst@fiu.edu}
\urladdr{http://www.fiu.edu/$\sim$lenesst}
\dedicatory{}
\thanks{Paul Feehan was partially supported by National Science Foundation grant DMS-1510064 and Thomas Leness was partially supported by National Science Foundation grant DMS-1510063.}
\begin{abstract}
We show that the $\SO(3)$ monopole cobordism formula from \cite{FL5}
implies that all closed, oriented, smooth four-manifolds with $b^1=0$ and $b^+\ge 3$
and odd with Seiberg--Witten simple type satisfy
the superconformal simple type condition defined by Mari\~no, Moore, and Peradze \cite{MMPhep,MMPdg}.
This implies the lower bound, conjectured by Fintushel and Stern \cite{FSCanonClass},
on the number of Seiberg--Witten basic classes in terms of topological data.
\end{abstract}

\date{This version: September 21, 2019, incorporating final galley proof corrections. \emph{Advances in Mathematics} (2019), https://doi.org/10.1016/j.aim.2019.106817}

\maketitle

\section{Introduction}
For a closed, four-manifold $X$ we will use the characteristic numbers
\begin{equation}
\label{eq:CharNumbers}
\begin{aligned}
c_1^2(X) &:= 2e(X)+3\si(X),
\\
\chi_h(X) &:= (e(X)+\si(X))/4,
\\
c(X) &:= \chi_h(X)-c_1^2(X),
\end{aligned}
\end{equation}
where $e(X)$ and $\si(X)$ are the Euler characteristic and signature of $X$.
We call a four-manifold {\em standard\/} if it is closed, connected, oriented, and smooth with $b^+(X)\ge 3$ and odd and $b^1(X)=0$.  We will write $Q_X$ for the intersection
form of $X$ on both $H_2(X;\ZZ)$ and $H^2(X;\ZZ)$, as in \cite[Definition 1.2.1]{GompfStipsicz}.

For a standard four-manifold $X$, the Seiberg--Witten invariants define a function,
$\SW_X:\Spinc(X)\to\ZZ$,
on the set of $\spinc$ structures on $X$.
The {\em Seiberg--Witten basic classes\/} of $X$, $B(X)$, are
the image under $c_1:\Spinc(X)\to H^2(X;\ZZ)$ of the support of $\SW_X$.
A manifold $X$ has {\em Seiberg--Witten simple type\/} if $K^2=c_1^2(X)$ for all $K\in B(X)$.
All known standard four-manifolds have Seiberg--Witten simple type
(see \cite[Conjecture 1.6.2]{KMBook}).

Following \cite{MMPdg, MMPhep}, one says that a standard
four-manifold $X$ has {\em superconformal simple type\/} if
$c(X)\le 3$ or
for $w\in H^2(X;\ZZ)$ characteristic and $c(X)\ge 4$,
\begin{equation}
\label{eq:SWPolynomial}
\SW_X^{w,i}(h)
:=
\sum_{\fs\in\Spinc(X)}(-1)^{\half(w^{2}+c_1(\fs)\cdot w)}
SW_X(\fs)\langle c_1(\fs),h\rangle^i
=
0,
\end{equation}
for $i\le c(X)-4$ and all $h\in H_2(X;\RR)$.
Mari\~no, Moore, and Peradze  conjectured
that all standard four-manifolds of Seiberg--Witten simple type satisfy this condition \cite[Conjecture 7.8.1]{MMPhep}.

In \cite{FKLM}, we showed that if $X$ was {\em abundant\/} in the sense
that $B(X)^\perp$ (the orthogonal complement with respect to
$Q_X$) contained a hyperbolic summand, then $X$
had superconformal simple type.  In this article, we establish the following.

\begin{thm}
\label{thm:SCST}
Let $X$ be a standard four-manifold of Seiberg--Witten simple type and assume Hypothesis \ref{hyp:Local_gluing_map_properties}. Then $X$ has superconformal simple type.
\end{thm}

In \cite{FL5}, we proved the required $\SO(3)$-monopole link-pairing formula, restated in this article as Theorem \ref{thm:SWLinkPairing}, assuming the validity of certain technical properties --- comprising Hypothesis \ref{hyp:Local_gluing_map_properties} and described in more detail in Remark \ref{rmk:GluingThmProperties} --- of the local gluing maps for $\SO(3)$ monopoles constructed in \cite{FL3}. A proof of the required local $\SO(3)$-monopole gluing-map properties, which may be expected from known properties of local gluing maps for anti-self-dual $\SO(3)$ connections and Seiberg--Witten monopoles, is currently being developed by the authors \cite{Feehan_Leness_monopolegluingbook}.

One might draw a comparison between our use of the $\SO(3)$-monopole link-pairing formula in our proof of Theorem \ref{thm:SCST} and G\"ottsche's assumption of the validity of the Kotschick--Morgan Conjecture \cite{KotschickMorgan} in his proof \cite{Goettsche} of the wall-crossing formula for Donaldson invariants. However, such a comparison overlooks the fact that our assumption of certain properties for local $\SO(3)$-monopole gluing maps is narrower and more specific. Indeed, our monograph \cite{FL5} effectively contains a proof of the Kotschick--Morgan Conjecture, modulo the assumption of certain technical properties for local gluing maps for anti-self-dual $\SO(3)$ connections which extend previous results of Taubes \cite{TauSelfDual, TauIndef, TauFrame}, Donaldson and Kronheimer \cite{DK}, and Morgan and Mrowka \cite{MorganMrowkaTube, MrowkaThesis}. Our proof of Theorem \ref{thm:SWLinkPairing} in \cite{FL5} relies on our construction of a global gluing map for $\SO(3)$ monopoles and that in turn builds on properties of local gluing maps for $\SO(3)$ monopoles; the analogous comments apply to the proof of the Kotschick--Morgan Conjecture.

\subsection{Background and applications}
In \cite{MMPdg, MMPhep},
Mari\~no, Moore, and Peradze originally defined
the concept of superconformal simple type in the context
of supersymmetric quantum field theory.  With those methods, they
argued that a four-manifold satisfying the superconformal simple type
condition also satisfied the vanishing result for low degree
terms of the Seiberg--Witten series given in \eqref{eq:SWPolynomial}.
Because of the applications of \eqref{eq:SWPolynomial} described here, we use
\eqref{eq:SWPolynomial} as the definition of superconformal simple type.
Not only do all known examples of
four-manifolds satisfy this definition, but the condition is preserved under the standard surgery operations
(blow-up, torus sum, and rational blow-down) used to construct new examples (see \cite[Section 7]{MMPhep}).
The article \cite{FKLM} establishes that abundant four-manifolds have superconformal
simple type, but also provides  an example of a non-abundant four-manifold which still has
superconformal simple type.  Hence, the results established here are strictly stronger than those of \cite{FKLM}.

Mochizuki \cite{Mochizuki_2009} proved a formula (see Theorem 4.1 in \cite{Goettsche_Nakajima_Yoshioka_2011}) which expresses the Donaldson invariants of a complex projective surface
in a form similar to that given by the $\SO(3)$-monopole cobordism formula
\cite[Theorem 1]{FL5},
but with coefficients
given as the residues of an explicit $\CC^*$-equivariant integral over the product of Hilbert schemes of points on $X$.
In \cite{Goettsche_Nakajima_Yoshioka_2011}, G\"ottsche, Nakajima, and Yoshioka showed how Witten's
Conjecture (given here as Conjecture \ref{conj:WittenSimpleType}) followed from Mochizuki's formula.
In addition, they conjectured \cite[Conjecture 4.5]{Goettsche_Nakajima_Yoshioka_2011} that
Mochizuki's formula (and hence their proof of Witten's Conjecture)
holds for all standard four-manifolds and not just complex projective surfaces.
Their \cite[Proposition 8.9]{Goettsche_Nakajima_Yoshioka_2011}
shows that all four-manifolds satisfying Mochizuki's formula have superconformal simple type.
The development in \cite{Goettsche_Nakajima_Yoshioka_2011} relies on Mochizuki's formula for the Donaldson
invariant and that is conjectured in \cite{Goettsche_Nakajima_Yoshioka_2011} to be equivalent to the version of the $\SO(3)$-monopole cobordism formula
given in
\cite[Theorem 1]{FL5}.
In contrast, this article uses a version of the $\SO(3)$-monopole
cobordism formula, Theorem \ref{lem:LevelOfSWinLa0}, which does not involve the Donaldson invariant
and so the two proofs are quite different.    Using Mochizuki's techniques from \cite{Mochizuki_2009}
to find an equation similar to that in Theorem \ref{lem:LevelOfSWinLa0} and discovering what that
equation would imply about its coefficients poses an interesting question for future
research.


The superconformal simple type condition is not only relevant to physics and algebraic geometry.
Using the vanishing condition \eqref{eq:SWPolynomial} as a definition, in \cite[Theorem 8.1.1]{MMPhep},
Mari\~no, Moore, and Peradze rigorously derived a lower bound on the
number of basic classes for manifolds of superconformal simple type.
Theorem \ref{thm:SCST} and \cite[Theorem 8.1.1]{MMPhep} therefore yield a proof of the following result,
first conjectured by Fintushel and Stern \cite{FSCanonClass}.

\begin{cor}
\label{cor:NumberOfBasicClasses}
Let $X$ be a standard four-manifold of Seiberg--Witten simple type.
If $B(X)$ is non-empty and $c(X)\ge 3$, then the $\SO(3)$-monopole link-pairing
formula (Theorem \ref{thm:SWLinkPairing}) implies that $|B(X)/\{\pm 1\}|\ge [c(X)/2]$.
\end{cor}

Theorem \ref{thm:SCST} also completes a proof of the derivation of Witten's Conjecture
relating Donaldson and Seiberg--Witten invariants from the $\SO(3)$-monopole cobordism
formula of \cite{FL5}.
In \cite{Witten}, Witten conjectured the following relation between the
Seiberg--Witten and Donaldson invariants
(see \cite[Lemma 2.8]{FL6} for this equivalent form of the conjecture).

\begin{conj}[Witten's Conjecture]
\label{conj:WittenSimpleType}
\cite{Witten}
Let $X$ be
standard four-manifold with Seiberg--Witten simple type.
Then
for any $w\in H^2(X;\ZZ)$, $h\in H_2(X;\RR)$,
and positive generator $x\in H_0(X;\ZZ)$,
the Donaldson invariant, $D^w_X$, as defined in \cite{DonPoly,KMStructure}
satisfies
\begin{equation}
\label{eq:DInvarForWC}
D^w_X(h^{\delta-2m}x^m)
=
2^{2-c(X)}
\sum_{\begin{subarray}{l}i+2k\\=\delta-2m\end{subarray}}
\frac{(\delta-2m)!}{2^{k-m} k!i!}
\SW_X^{w,i}(h)Q_X(h)^k,
\end{equation}
when $\delta$ is a non-negative integer obeying $\delta \equiv -w^2 - 3\chi_h(X)\pmod{4}$.
\end{conj}

By definition, $D^w_X(h^{\delta-2m}x^m) = 0$ when $\delta$ is a non-negative integer that does not obey $\delta \equiv -w^2 - 3\chi_h(X)\pmod{4}$.
In \cite{FL5}, using the moduli space of $\SO(3)$ monopoles defined by Pidstrigatch and Tyurin \cite{PTLocal}
for this purpose and assuming the technical properties for local $\SO(3)$-monopole gluing maps described in Hypothesis \ref{hyp:Local_gluing_map_properties},
we proved the \emph{$\SO(3)$-monopole cobordism formula\/},
which expresses the Donaldson polynomial $D^w_X$
of \cite{DonPoly,KMStructure} as a polynomial in the Seiberg--Witten polynomials in
\eqref{eq:SWPolynomial}, the intersection form, $Q_X$, and an additional cohomology class $\La$
on $H_2(X;\RR)$,
\begin{equation}
\label{eq:RoughMCF}
D^w_X(h^{\delta-2m}x^m) =
\sum_{\begin{subarray}{l}i+j+2k\\=\delta-2m\end{subarray}}
\sum_{K\in B(X)}
 a_{i,j,k} \SW_X(K)\langle K,h\rangle^i \langle\La,h\rangle^j Q_X(h)^k
\end{equation}
where the real coefficients, $a_{i,j,k}$, are unknown but depend only on homotopy invariants of the four-manifold $X$.
It became apparent in \cite{FKLM,FL6}
that superconformal simple type functioned as an obstruction to determining
these coefficients.  That is, because the Seiberg--Witten polynomials $\SW^{w,i}_X$
vanish when $i\le c(X)-3$ for all known examples, we could not use examples where Witten's Conjecture held to determine
the coefficients $a_{i,j,k}$ with $i< c(X)-3$.
However, in \cite{FL7}, we showed that
while we could not determine the coefficients $a_{i,j,k}$ with
$i< c(X)-3$, we could show that they satisfied a difference
equation and by combining the superconformal simple type condition with this difference equation,
we could derive Witten's Conjecture from \eqref{eq:RoughMCF}.
Thus, Theorem \ref{thm:SCST} and the results of \cite{FL7} give the following

\begin{cor}
\label{cor:WC}
Let $X$ be a standard four-manifold and assume Hypothesis \ref{hyp:Local_gluing_map_properties}. Then Witten's Conjecture \ref{conj:WittenSimpleType} holds.
\end{cor}

Recall that Hypothesis \ref{hyp:Local_gluing_map_properties} refers to certain expected properties for local gluing maps for $\SO(3)$ monopoles.

\subsection{Outline}
After reviewing definitions and basic properties of the Seiberg--Witten invariants
in Section \ref{subsec:SWDefn}, we introduce the moduli space of
$\SO(3)$ monopoles in Section \ref{subsec:SO(3)Monopoles} and review results
from \cite{FL1,FL2a,FL2b,FL5} on the monopole cobordism formula.
We consider a particular case of this formula in Section \ref{sec:CobordWithLa=0}
to get,  in Theorem \ref{thm:VanishingCobordism},
a form of the cobordism formula where the pairing with the link of
the  moduli space of anti-self-dual connections vanishes by a dimension-counting argument.  This cobordism
formula then states that a sum over $K\in B(X)$ of pairings with links of the Seiberg--Witten
moduli space corresponding to $K$ vanishes, giving an equality of the form (see \eqref{eq:VanishingCobordismReducedForm})
\begin{equation}
\label{eq:MCFormula1}
0
=
\sum_{k=0}^\ell a_{c-2v+2k,0,\ell-k} \SW^{w,c-2v+2k}_X Q_X^{\ell-k},
\end{equation}
where we abbreviate $c=c(X)$.
In Section \ref{sec:LeadingOrderComp}, we show that the coefficient  $a_{c-2v,0,\ell}$
appearing in \eqref{eq:MCFormula1} is non-zero by applying the methods used in
\cite{KotschickMorgan} to the topological description of the link of the Seiberg--Witten
moduli space given in \cite{FL5}.  We show that the coefficients $a_{c-2v+2k,0,\ell-k}$
in \eqref{eq:MCFormula1} vanish if $c-2v+2k\ge c-3$
in Section \ref{sec:CoeffOfVanishing}.
In Section \ref{sec:MainProof}, we combine this information on the coefficients
and give an inductive argument proving Theorem
\ref{thm:SCST}.

\subsection{Acknowledgements}
The authors would like to thank Ron Fintushel, Inanc Baykur and Nikolai Saveliev for helpful discussions on examples of four-manifolds as well as Tom Mrowka and Simon Donaldson for their support of this project. We also thank the anonymous referees for their comments and careful reading of our manuscript. Paul Feehan is grateful for support from the National Science Foundation under grant DMS-1510064 and Thomas Leness is grateful for support from the National Science Foundation grant DMS-1510063.

\section{Preliminaries}
\label{sec:Prelim}

\subsection{Seiberg--Witten invariants}
\label{subsec:SWDefn}
Detailed expositions of  the Seiberg--Witten invariants,
introduced by Witten in \cite{Witten},
are provided in \cite{KMBook,MorganSWNotes,NicolaescuSWNotes}.
These invariants define a  map with finite support,
$$
\SW_X:\Spinc(X)\to\ZZ,
$$
from the set of \spinc structures on $X$.
A \spinc structure $\fs=(W^\pm,\rho_W)$ on $X$ consists of a pair of complex rank-two vector bundles,
$W^\pm\to X$, and a Clifford multiplication map, $\rho_W:T^*X\to \Hom_\CC(W^+,W^-)$.
If $\fs\in\Spinc(X)$, then $c_1(\fs):=c_1(W^+) \in H^2(X;\ZZ)$
is characteristic.

One calls $c_1(\fs)$ a {\em Seiberg--Witten basic class\/} if $\SW_X(\fs)\neq 0$.
Define
\begin{equation}
\label{eq:SWBasic}
B(X):=
\{c_1(\fs): \SW_X(\fs)\neq 0\}.
\end{equation}
If $H^2(X;\ZZ)$ has 2-torsion, then $c_1:\Spinc(X)\to H^2(X;\ZZ)$ is not injective.
Because we will work with functions involving real homology and cohomology, we  define
\begin{equation}
\label{eq:DefineCohomSW}
\SW_X':H^2(X;\ZZ)\to\ZZ,
\quad
K\mapsto\sum_{\fs\in c_1^{-1}(K)}\SW_X(\fs).
\end{equation}
Thus, we can rewrite
the expression for $\SW_X^{w,i}(h)$ in \eqref{eq:SWPolynomial} as
\begin{equation}
\label{eq:ReduceSWPolyToK}
\SW_X^{w,i}(h)=\sum_{K\in B(X)} (-1)^{\thalf(w^2+w\cdot K)}\SW_X'(K)\langle K,h\rangle^i.
\end{equation}
A four-manifold $X$ has {\em Seiberg--Witten simple type\/} if $\SW_X(\fs)\neq 0$ implies that $c_1(\fs)^2=c_1^2(X)$.

\subsection{$\SO(3)$ monopoles}
\label{subsec:SO(3)Monopoles}
We now review the basic definitions and results on the moduli space of
$\SO(3)$ monopoles.  More detailed discussions of these results can be
found in \cite{FL2a,FL2b}.
\subsubsection{\Spinu\ structures}
A \spinu structure $\ft=(V^\pm,\rho)$ on a four-manifold $X$ is a pair of complex rank-four
vector bundles $V^\pm\to X$ with a Clifford module structure $\rho:T^*X\to\Hom_\CC(V^+,V^-)$.
In more familiar terms, for a \spinc structure $\fs=(W^\pm,\rho_W)$ on $X$,
a \spinu structure is given by $V^\pm=W^\pm\otimes E$, where $E\to X$ is a complex
rank-two vector bundle and the Clifford multiplication map is given by $\rho=\rho_W\otimes \id_E$.
We define characteristic classes of a \spinu structure $\ft=(W^\pm\otimes E,\rho)$ by
$$
p_1(\ft):=p_1(\su(E)),
\quad
c_1(\ft):=c_1(W^+)+c_1(E),
\quad
w_2(\ft):=c_1(E)\pmod 2.
$$

\begin{lem}
\label{lem:ExistenceOfSpinu}
Let $X$ be a standard four-manifold.
Given $\wp\in H^4(X;\ZZ)$, $\La\in H^2(X;\ZZ)$, and $\mathfrak{w}\in H^2(X;\ZZ/2\ZZ)$, there is a \spinu
structure $\ft$ on $X$ with $p_1(\ft)=\wp$, $c_1(\ft)=\La$, and $w_2(\ft)=\mathfrak{w}$
if and only if:
\begin{enumerate}
\item
There is a class $w\in H^2(X;\ZZ)$ with $\mathfrak{w}=w\pmod 2$,
\item
$\La\equiv \mathfrak{w}+w_2(X)\pmod 2$,
\item
$\wp\equiv \mathfrak{w}^2\pmod 4$.
\end{enumerate}
\end{lem}

\begin{proof}
Given $(\wp,\La,\mathfrak{w})$ and $w$ satisfying the three conditions above,
we observe that $\La-w$ is characteristic so there is a \spinc structure
$\fs=(W^\pm,\rho_W)$ with $c_1(\fs)=\La-w$.  Let $E\to X$ be the rank-two complex vector
bundle with $c_1(E)=w$ and $c_2(E)=(w^2-\wp)/4$.  Define $\ft$ by $V^\pm=W^\pm\otimes E$
and $\rho=\rho_W\otimes\id_E$.
Observe that $p_1(\su(E))=c_1(E)^2-4c_2(E)=w^2-w^2+\wp=\wp$ and
$w_2(\ft)=w_2(\su(E))\equiv c_1(E)\equiv \mathfrak{w}\pmod 2$,
while $c_1(\ft)=c_1(E)+c_1(\fs)=w+\La-w=\La$, as required.

Given a \spinu structure $\ft$, these properties of its characteristic classes
follow from easy computations.
\end{proof}

\subsubsection{The moduli space of $\SO(3)$ monopoles and fixed points of a circle action}
For a \spinu structure $\ft=(W^\pm\otimes E,\rho)$ on $X$, the moduli space of $\SO(3)$ monopoles
on $\ft$ is the space
of solutions, modulo gauge equivalence, to the $\SO(3)$-monopole equations
(namely, \cite[Equation (1.1)]{FL1} or \cite[Equation (2.32)]{FL2a})
for a pair $(A,\Phi)$
where $A$ is a unitary connection on $E$ and $\Phi\in\Om^0(V^+)$.
We write this moduli space as $\sM_\ft$.

Complex scalar multiplication on the section, $\Phi$, defines an $S^1$ action on
$\sM_\ft$ with stabilizer $\{\pm 1\}$ away from two families of fixed point sets:
\begin{inparaenum}[(1)]
\item zero-section points $[A,0]$, and
\item reducible points $[A,\Phi]$, where
$A$ is reducible.
\end{inparaenum}

By \cite[Section 3.2]{FL2a}, the subspace
of zero-section points
is a manifold with a natural smooth structure diffeomorphic to
the moduli space of anti-self-dual
connections on $\su(E)$ which we denote,
following the notation of \cite{KMStructure}, by $M^w_\ka$ where $\ka=-p_1(\ft)/4$
and $w=c_1(E)$.

By \cite[Lemma 2.17]{FL2a},
the subspace of reducible points where $A$ is
reducible with respect to a splitting $E= L_1\oplus L_2$
is a manifold,
$M_\fs$, which is compactly cobordant to the moduli space of Seiberg--Witten monopoles
associated with the \spinc structure $\fs$ where $c_1(\fs)=c_1(W^+\otimes L_1)$.
By \cite[Lemma 3.32]{FL2b}, the possible splittings of $\ft$ are given by
$$
\Red(\ft)=
\{\fs\in\Spinc:\left( c_1(\fs)-c_1(\ft)\right)^2
=
p_1(\ft) \}.
$$
Hence, the subspace of reducible points is
$$
M^{\red}_\ft=\bigcup_{\fs\in\Red(\ft)} M_\fs.
$$
We define
$$
\sM^0_\ft:=\sM_\ft\setminus M^w_\ka,
\quad
\sM^*_\ft:=\sM_\ft\setminus M^{\red}_\ft,
\quad\hbox{and}\quad
\sM^{*,0}_\ft:=\sM^0_\ft\cap \sM^*_\ft.
$$
We recall the

\begin{thm}
\cite{FeehanGenericMetric,FL1,TelemanGenericMetric}
Let $\ft$ be a \spinu structure on a standard four-manifold $X$.
For generic perturbations of the $\SO(3)$ monopole equations,
the moduli space $\sM^{*,0}_\ft$ is a smooth, orientable manifold
of dimension
$$
\dim\sM_\ft= 2d_a(\ft)+ 2n_a(\ft),
$$
where
\begin{subequations}
\label{eq:DefineIndex}
\begin{align}
d_a(\ft)&:=\frac{1}{2}\dim M^w_\ka=-p_1(\ft)-3\chi_h(X), \label{eq:DefineIndex1}
\\
n_a(\ft)&:=\frac{1}{4}\left( p_1(\ft)+c_1(\ft)^2 -c_1^2(X)+8\chi_h(X)\right), \label{eq:DefineIndex2}
\end{align}
\end{subequations}
with $\chi_h(X)$ and $c_1^2(X)$ as in \eqref{eq:CharNumbers}.
\end{thm}

\subsubsection{The compactification}
\label{subsubsec:Compactification}
The moduli space $\sM_\ft$ is not compact but admits an Uhlenbeck compactification
as follows (see \cite[Section 2.2]{FL2a} or \cite{FL1} for details).
For $\ell\ge 0$, let $\ft(\ell)$ be the \spinu structure on $X$ with
$$
p_1(\ft(\ell))=p_1(\ft)+4\ell,
\quad
c_1(\ft(\ell))=c_1(\ft),
\quad
w_2(\ft(\ell))=w_2(\ft).
$$
Let $\Sym^\ell(X)$ be the $\ell$-th symmetric product of $X$ (that is, $X^\ell$ modulo
the symmetric group on $\ell$ elements).  For $\ell=0$, we define $\Sym^\ell(X)$ to
be a point.  The space of ideal $\SO(3)$ monopoles is defined by
$$
I\sM_\ft
:=
\bigcup_{\ell=0}^N \sM_{\ft(\ell)}\times\Sym^\ell(X).
$$
We give $I\sM_\ft$ the topology defined by Uhlenbeck convergence
(see \cite[Definition 4.9]{FL1}).

\begin{thm}
\cite{FL1}
Let $X$ be a standard four-manifold with Riemannian metric, $g$.
Let
$\bar\sM_\ft\subset I\sM_\ft$ be the closure of $\sM_\ft$ with
respect to the Uhlenbeck topology.
Then there is a non-negative integer $N$ depending only on $(X,g)$, $p_1(\ft)$, and $c_1(\ft)$
such that $\bar\sM_\ft$ is compact.
\end{thm}

The $S^1$ action on $\sM_\ft$ extends continuously over $I\sM_\ft$
and $\bar\sM_\ft$, in particular, but $\bar\sM_\ft$ contains fixed points of this $S^1$ action which are not contained in $\sM_\ft$.
The closure of $M^w_\ka$ in $\bar\sM_\ft$ is $\bar M^w_\ka$, the
Uhlenbeck compactification of the moduli space of
anti-self-dual connections as defined in \cite{DK}.
There are additional reducible points in the lower
levels
of $I\sM_\ft$.
Define
$$
\overline{\Red}(\ft)
:=
\{\fs\in\Spinc(X): \left( c_1(\fs)-c_1(\ft)\right)^2\ge p_1(\ft)\}.
$$
If we define the {\em level\/} $\ell(\ft,\fs)$ in $\bar\sM_\ft$ of the
\spinc structure $\fs$ by
\begin{equation}
\label{eq:DefineLevelOfRed}
\ell(\ft,\fs)
:=
\frac{1}{4}\left(  \left( c_1(\fs)-c_1(\ft)\right)^2- p_1(\ft)\right),
\end{equation}
then the
set of
strata of reducible points in $\bar\sM_\ft$ are given by
\begin{equation}
\bar M^{\red}_\ft
:=
\bigcup_{\fs\in \overline{\Red}(\ft)} M_\fs\times\Sym^{\ell(\ft,\fs)}(X).
\end{equation}
Note that for $\fs\in\overline{\Red}(\ft)$, we have $\ell(\ft,\fs)\ge 0$ by
the definitions of $\overline{\Red}(\ft)$ and $\ell(\ft,\fs)$.
By analogy with the corresponding definitions for $\sM_\ft$, we write
$$
\bar\sM^0_\ft:=\bar\sM_\ft\setminus \bar M^w_\ka,
\quad
\bar\sM^*_\ft:=\bar \sM_\ft\setminus \bar M^{\red}_\ft,
\quad
\bar \sM^{*,0}_\ft:=\bar\sM^0_\ft\cap \bar\sM^*_\ft,
$$
and observe that the stabilizer of the $S^1$ action on $\bar\sM^{*,0}_\ft$ is $\{\pm 1\}$.

\subsubsection{Cohomology classes and geometric representatives}
\label{subsubsec:CohomClass}
The cohomology classes used to define Donaldson invariants extend
to $\sM^*_\ft/S^1$.
For
$\beta\in H_\bullet(X;\RR)$,
there is a cohomology class,
$$
\mu_p(\beta)\in H^{4-\bullet}(\sM^*_\ft/S^1;\RR),
$$
with geometric representative (in the sense of \cite[p. 588]{KMStructure}
or \cite[Definition 3.4]{FL2b}),
$$
\sV(\beta)\subset \sM^*_\ft/S^1.
$$
For $h_i\in H_2(X;\RR)$ and a generator $x\in H_0(X;\ZZ)$,
we define
\begin{align*}
\mu_p(h_1\cdots h_{\delta-2m}x^m)
{}&=
\mu_p(h_1)\smile\dots\smile \mu_p(h_{\delta-2m})\smile\barmu_p(x)^m
\in
H^{2\delta}(\sM^*_\ft/S^1;\RR),
\\
\sV(h_1\cdots h_{\delta-2m}x^m)
{}&=
\sV(h_1)\cap\dots \cap\sV(h_{\delta-2m})\cap
\underbrace{\sV(x)\cap\dots\cap\sV(x)}_{\text{$m$ copies}},
\end{align*}
and let $\bar\sV(h_1\cdots h_{\delta-2m}x^m)$ be the closure
of $\sV(h_1\cdots h_{\delta-2m}x^m)$ in $\bar\sM^*_\ft/S^1$.

Denote
the first Chern class of the $S^1$ action on $\bar\sM^{*,0}_\ft$
with multiplicity two by
$$
\bar\mu_c\in H^2(\bar\sM^{*,0}_\ft/S^1;\ZZ).
$$
This cohomology class has a
geometric representative $\bar\sW$.

\subsubsection{The link of the moduli space of anti-self-dual connections}
Let $\bL^{\asd}_\ft$ be the
link of $\bar M^w_\ka\subset\bar\sM_\ft/S^1$ (see \cite[Definition 3.7]{FL2a}).
The space $\bL^{\asd}_\ft$ is stratified by smooth manifolds,
with
lower strata of codimension at least two.
The top stratum of $\bL^{\asd}_\ft$ is a smooth, codimension-one submanifold
of $\sM^{*,0}_\ft/S^1$ and so has dimension twice
\begin{equation}
\label{eq:ASDLinkDim}
\frac{1}{2} \dim \bL^{\asd}_\ft
=
d_a(\ft)+n_a(\ft)-1.
\end{equation}
Just as an integral lift $w$ of $w_2(\ft)$ defines an orientation
for $M^w_\ka$ in \cite{DonOrient}, the choice of $w$ defines a compatible orientation
for the top stratum of $\bL^{\asd}_\ft$ (see \cite[Lemma 3.27]{FL2b}).
The intersection of the geometric representatives in Section \ref{subsubsec:CohomClass}
with $\bL^{\asd}_\ft$
can be used to compute Donaldson invariants \cite{DonPoly,KMStructure}
or spin polynomial invariants  \cite{PTDirac}.
We will need the following vanishing result.
We note that $\NN = \{0,1,2,\ldots\}$ denotes the set of non-negative integers here and throughout the remainder of our article.

\begin{prop}
\label{prop:VanishingASDPairing}
\cite[Proposition 3.29]{FL2b}
Let $\ft$ be a \spinu structure on a standard four-manifold $X$.
For $\delta,\eta_c,m\in\NN$, if
\begin{subequations}
\begin{align}
\label{eq:VanishingASDCondition1}
\delta-2m\ge 0,
\\
\label{eq:VanishingASDCondition2}
\delta+\eta_c=\frac{1}{2}\dim \bL^{\asd}_\ft=d_a(\ft)+n_a(\ft)-1,
\\
\label{eq:VanishingASDCondition3}
\delta>\frac{1}{2}\dim M^w_\ka=d_a(\ft)\ge 0,
\end{align}
\end{subequations}
then
$$
\#\left(\bar\sV(h^{\delta-2m}x^m)\cap \bar\sW^{\eta_c}\cap \bL^{\asd}_\ft\right)
=0,
$$
where $\#$ denotes the signed count of the points in the intersection.
\end{prop}

\subsubsection{The links of the moduli spaces of Seiberg--Witten monopoles}
For $\ell(\ft,\fs)\ge 0$,
the link $\bL_{\ft,\fs}$ of $M_\fs\times\Sym^\ell(X)\subset \bar\sM_\ft/S^1$
is defined in \cite{FL5}.  The space $\bL_{\ft,\fs}$ is compact, stratified by smooth
manifolds with corners, with lower strata of codimension at least two.
The dimension of $\bL_{\ft,\fs}$ equals that of $\bL^{\asd}_\ft$.
As described in
\cite[Section 8.1.4]{FL5}, the top stratum of $\bL_{\ft,\fs}$ is orientable with a natural choice of orientation.

\begin{hyp}[Properties of local $\SO(3)$-monopole gluing maps]
\label{hyp:Local_gluing_map_properties}
The local gluing map, constructed in \cite{FL3}, gives a continuous parametrization of a neighborhood of $M_{\fs}\times\Si$ in $\bar\sM_{\ft}$ for each smooth stratum $\Si\subset\Sym^\ell(X)$.
\end{hyp}

Hypothesis \ref{hyp:Local_gluing_map_properties} is discussed in greater detail in
\cite[Sections 7.8 \& 7.9]{FL5}.
The question of how to assemble the \emph{local} gluing maps for neighborhoods of $M_{\fs}\times \Si$ in $\bar\sM_{\ft}$, as $\Si$ ranges over all smooth strata of $\Sym^\ell(X)$, into a \emph{global} gluing map for a neighborhood of $M_{\fs}\times \Sym^\ell(X)$ in $\bar\sM_{\ft}$ is itself difficult --- involving the so-called `overlap problem' described in \cite{FLMcMaster} --- but one which we do solve in \cite{FL5}. See Remark \ref{rmk:GluingThmProperties} for a further discussion of this point.

\begin{thm}[$\SO(3)$-monopole link pairing formula]
\cite[Theorem 10.1.1]{FL5}
\label{thm:SWLinkPairing}
Let $\ft$ be a \spinu \hfill\break
structure on a standard four-manifold
$X$ of Seiberg--Witten simple type and assume Hypothesis \ref{hyp:Local_gluing_map_properties}.
Denote $\La=c_1(\ft)$ and $K=c_1(\fs)$
for $\fs\in\overline{\Red}(\ft)$.
Let $\delta,\eta_c,m\in\ZZ_{\ge 0}$ satisfy $\delta-2m\ge 0$
and
$$
\delta+\eta_c= \frac{1}{2}\dim \bL_{\ft,\fs}=d_a(\ft)+n_a(\ft)-1.
$$
Let $\ell=\ell(\ft,\fs)$  be as defined in
\eqref{eq:DefineLevelOfRed}.  Then,
for any integral lift $w\in H^2(X;\ZZ)$ of $w_2(\ft)$,
and any $h\in H_2(X;\RR)$, and
generator $x\in H_0(X;\ZZ)$,
\begin{equation}
\label{eq:MainEquation}
\begin{aligned}
{}&
\#\left(\bar\sV(h^{\delta-2m}x^m)\cap \bar\sW^{\eta_c}\cap \bL_{\ft,\fs}\right)
\\
{}&
\quad=
SW_X(\fs)
\sum_{\begin{subarray}{l}i+j+2k\\=\delta-2m\end{subarray}}
a_{i,j,k}(\chi_h,c_1^2,K\cdot\La,\La^2,m,\ell)
\langle K,h\rangle^i
\langle \La,h\rangle^j
Q_X(h)^k,
\end{aligned}
\end{equation}
where $\#$ denotes the signed count of points in the intersection
and
where for each triple of non-negative integers, $i, j, k \in \NN$,
the coefficients,
$$
a_{i,j,k}:\ZZ\times\ZZ\times\ZZ\times\ZZ\times\NN\times\NN \to \RR,
$$
are
universal
functions of the variables $\chi_h$, $c_1^2$, $c_1(\fs)\cdot\La$, $\La^2$, $m$, $\ell$
and vanish if $k>\ell(\ft,\fs)$.
\end{thm}

\begin{rmk}
\label{rmk:AdditionalParameterInCoeff}
In contrast to the version of this theorem presented in \cite{FL6},
the coefficients $a_{i,j,k}$  \eqref{eq:MainEquation} depend on the additional argument
$\ell$ because we do not assume that $\delta=\frac{1}{2}\dim M^w_\ka$ in \eqref{eq:MainEquation} as we do
in \cite{FL6}.
\end{rmk}

\begin{rmk}
\label{rmk:GluingThmProperties}
The proof in \cite{FL5} of Theorem \ref{thm:SWLinkPairing} assumes the Hypothesis \ref{hyp:Local_gluing_map_properties} (see
\cite[Sections 7.8 \& 7.9]{FL5}) that the local gluing map for a neighborhood of $M_{\fs}\times \Si$ in $\bar\sM_{\ft}$ gives a continuous parametrization of a neighborhood of $M_{\fs}\times\Si$ in $\bar\sM_{\ft}$ for each smooth stratum $\Si\subset\Sym^\ell(X)$. These local gluing maps are the analogues for $\SO(3)$ monopoles of the local gluing maps for anti-self-dual $\SO(3)$ connections constructed by Taubes in \cite{TauSelfDual, TauIndef, TauFrame} and Donaldson and Kronheimer in \cite[\S 7.2]{DK}; see also \cite{MorganMrowkaTube, MrowkaThesis}. We have established the existence of local gluing maps in \cite{FL3} and expect that a proof of the continuity for the local gluing maps with respect to Uhlenbeck limits should be similar to our proof in \cite{FLKM1} of this property for the local gluing maps for anti-self-dual $\SO(3)$ connections. The remaining properties of local gluing maps assumed in \cite{FL5} are that they are injective and also surjective in the sense that elements of $\bar\sM_{\ft}$ sufficiently close (in the Uhlenbeck topology) to $M_{\fs}\times\Si$ are in the image of at least one of the  local gluing maps. In special cases, proofs of these properties for the local gluing maps for anti-self-dual $\SO(3)$ connections (namely, continuity with respect to Uhlenbeck limits, injectivity, and surjectivity) have been given in \cite[\S 7.2.5, 7.2.6]{DK}, \cite{TauSelfDual, TauIndef, TauFrame}. The authors are currently developing a proof of the required properties for the local gluing maps for $\SO(3)$ monopoles \cite{Feehan_Leness_monopolegluingbook}.
Our proof will also yield the analogous properties for the local gluing maps for anti-self-dual $\SO(3)$ connections, as required to complete the proof of the Kotschick--Morgan Conjecture \cite{KotschickMorgan}, based on our work in \cite{FL5}.
\end{rmk}

\subsubsection{The cobordism formula}
The compactification $\bar\sM^{*,0}_\ft/S^1$ defines a compact
cobordism,
stratified by smooth oriented manifolds,   between
$$
\bL^{\asd}_\ft
\quad\text{and}\quad
\bigcup_{\fs\in\overline{\Red}(\ft)} \bL_{\ft,\fs}.
$$
For $\delta+\eta_c=\frac{1}{2}\dim\bL^{\asd}_\ft$, this cobordism gives the
following equality
\cite[Equation (2.6.1)]{FL5},
\begin{equation}
\label{eq:CobordismFormulaIntForm}
\begin{aligned}
{}&\#\left(\bar\sV(h^{\delta-2m}x^m)\cap \bar\sW^{\eta_c}\cap \bL^{\asd}_\ft\right)
\\
{}&\quad =
-\sum_{\fs\in\overline{\Red}(\ft)}
(-1)^{\frac{1}{2}(w^2-\si)+\frac{1}{2}(w^2+(w-c_1(\ft))\cdot c_1(\fs))}
\#\left(\bar\sV(h^{\delta-2m}x^m)\cap \bar\sW^{\eta_c}\cap \bL_{\ft,\fs}\right).
\end{aligned}
\end{equation}
We note that the power of $-1$ in \eqref{eq:CobordismFormulaIntForm} is computed by comparing
the different orientations of the links
as described in
\cite[Lemma 8.1.8]{FL5}.

\section{The cobordism with $c_1(\ft)=0$}
\label{sec:CobordWithLa=0}
In this section, we will derive a formula (see \eqref{eq:VanishingCobordismReducedForm})
relating the Seiberg--Witten polynomials $\SW^{w,i}_X$ defined in \eqref{eq:SWPolynomial}
and the intersection form of $X$.
We do so by  applying the cobordism formula \eqref{eq:CobordismFormulaIntForm} in a case where
Proposition \ref{prop:VanishingASDPairing} implies that the left-hand-side
of \eqref{eq:CobordismFormulaIntForm} vanishes.
To extract a formula
from the resulting vanishing sum that includes the Seiberg--Witten
polynomials, $\SW^{w,i}_X$, we apply Theorem \ref{thm:SWLinkPairing} to
the terms on the right-hand-side of \eqref{eq:CobordismFormulaIntForm}.
In the resulting sum over
$\overline{\Red}(\ft)$, the coefficients, $a_{i,j,k}$, appearing
in equation \eqref{eq:MainEquation} in Theorem \ref{thm:SWLinkPairing} depend on $c_1(\ft)\cdot c_1(\fs)$.
This dependence prevents the
desired extraction of $\SW^{w,i}_X$ (see Remark \ref{rmk:VaryingCoeff}) from the cobordism sum.
To ensure that $c_1(\ft)\cdot c_1(\fs)$ is constant as $c_1(\fs)$ varies
in $B(X)$ without further assumptions on $B(X)$, such as the abundance
condition mentioned in our Introduction, we assume $c_1(\ft)=0$.


We begin by establishing the existence of a family of \spinu structures
with $c_1(\ft)=0$.

\begin{lem}
\label{lem:CharClassOfSpinuWithLa0}
Let $X$ be a standard four-manifold.
For every $n\in\NN$ there is a
\spinu structure $\ft_n$ on $X$ satisfying
\begin{equation}
\label{eq:CharClassOfSpinuWithLa0}
c_1(\ft_n)=0, \quad
p_1(\ft_n)=4n+c_1^2(X)-8\chi_h(X),\quad
w_2(\ft_n)=w_2(X),
\end{equation}
and such that $n_a(\ft_n)=n$, where $n_a(\ft)$ is the index defined in \eqref{eq:DefineIndex2}.
\end{lem}

\begin{proof}
By
\cite[Exercise 1.2.23]{GompfStipsicz}, $w_2(X)$ admits an integral lift.
Therefore, the existence of the \spinu structure $\ft_n$
with the characteristic classes in \eqref{eq:CharClassOfSpinuWithLa0}
follows from Lemma \ref{lem:ExistenceOfSpinu}
and the observation that for $c_1(\ft_n)=0$ and $w_2(\ft_n)=w_2(X)$ we have
$w_2(\ft_n)^2\equiv \si(X)\equiv c_1^2(X)-8\chi_h(X)\pmod 4$, so the desired value of $p_1(\ft_n)$
can be achieved for any $n \in\ZZ$ with $n\ge 0$.
The equality $n_a(\ft_n)=n$ follows from \eqref{eq:DefineIndex2} and by substituting
the value of $p_1(\ft_n)$ in \eqref{eq:CharClassOfSpinuWithLa0}.
\end{proof}

To apply Theorem \ref{thm:SWLinkPairing} to the cobordism formula \eqref{eq:CobordismFormulaIntForm}
for a \spinu structure $\ft_n$ satisfying \eqref{eq:CharClassOfSpinuWithLa0},
we compute the level in $\bar\sM_{\ft_n}$ of a \spinc structure $\fs$.

\begin{lem}
\label{lem:LevelOfSWinLa0}
Let $X$ be a standard four-manifold of Seiberg--Witten simple type.
For a non-negative integer $n$, let $\ft_n$ be a
\spinu structure on $X$ satisfying \eqref{eq:CharClassOfSpinuWithLa0}.
For $c_1(\fs)\in B(X)$,
the level $\ell=\ell(\ft_n,\fs)$ in $\bar\sM_{\ft_n}$ of $\fs$
is
\begin{equation}
\label{eq:LevelOfSWInLa=0}
\ell(\ft_n,\fs)
=
2\chi_h(X)-n.
\end{equation}
\end{lem}

\begin{proof}
By the definition of Seiberg--Witten simple type,
for any $c_1(\fs)\in B(X)$ we have
\begin{equation}
\label{eq:SWST}
c_1(\fs)^2= c_1^2(X).
\end{equation}
By \eqref{eq:DefineLevelOfRed}, the level is given by
\begin{align*}
\ell(\ft_n,\fs)
{}&=
\frac{1}{4}
\left(
(c_1(\fs)-c_1(\ft_n))^2 - p_1(\ft_n)
\right)
\\
{}&=
\frac{1}{4}
\left(
c_1(\fs)^2 - 4n -c_1^2(X)+8\chi_h(X)
\right)
\quad\text{(by  \eqref{eq:CharClassOfSpinuWithLa0})}
\\
{}&=
2\chi_h(X)-n
\quad\text{(by \eqref{eq:SWST})},
\end{align*}
as desired.
\end{proof}

Combining \eqref{eq:CobordismFormulaIntForm} with Proposition \ref{prop:VanishingASDPairing}
and Theorem \ref{thm:SWLinkPairing} then gives the following

\begin{thm}
\label{thm:VanishingCobordism}
Let $X$ be a standard four-manifold of Seiberg--Witten simple type.
Assume that $m,n\in\NN$ satisfy
\begin{subequations}
\label{eq:VanishingSumAssump}
\begin{align}
\label{eq:VanishingSumAssump1}
n&\le 2\chi_h(X),
\\
\label{eq:VanishingSumAssump2}
1 &< n,
\\
\label{eq:VanishingSumAssump3}
0&\le c(X)-n -2m-1.
\end{align}
\end{subequations}
We abbreviate the coefficients  in equation \eqref{eq:MainEquation} in Theorem \ref{thm:SWLinkPairing} by
\begin{equation}
\label{eq:CoeffAbbreviation}
a_{i,0,k}:=a_{i,0,k}(\chi_h(X),c_1^2(X),0,0,m,2\chi_h(X)-n).
\end{equation}
Then, for $A=c(X)-n-2m-1$ and $w\in H^2(X;\ZZ)$ characteristic,
\begin{equation}
\label{eq:VanishingCobordismReducedForm}
0=
\sum_{k=0}^{2\chi_h(X)-n}
a_{A+2k,0,2\chi_h(X)-n-k}
\SW_X^{w,A+2k}(h)
Q_X(h)^{2\chi_h(X)-n-k}.
\end{equation}
\end{thm}

\begin{proof}
Let $\ft_n$ be a \spinu structure
on $X$ satisfying \eqref{eq:CharClassOfSpinuWithLa0}, where $n$ is
the non-negative integer in the statement of Theorem \ref{thm:VanishingCobordism}.
The value of $p_1(\ft_n)$ in \eqref{eq:CharClassOfSpinuWithLa0}
and the expression for $d_a(\ft_n)$ given in \eqref{eq:DefineIndex1} and $c(X)$ in \eqref{eq:CharNumbers}
imply that
\begin{equation}
\label{eq:ASDDimenForLa=0}
\frac{1}{2}\dim M^w_\ka = d_a(\ft_n)=c(X)+4\chi_h(X)-4n.
\end{equation}
The value of $d_a(\ft_n)$ in \eqref{eq:ASDDimenForLa=0},
the equality $n_a(\ft_n)=n$ given in Lemma \ref{lem:CharClassOfSpinuWithLa0},
and the
formula for half the dimension of $\bL^{\asd}_{\ft_n}$ given in \eqref{eq:ASDLinkDim}
imply that
\begin{equation}
\label{eq:ASDLinkDimenForLa=0}
\frac{1}{2}\dim\bL^{\asd}_{\ft_n} = c(X)+4\chi_h(X)-3n-1.
\end{equation}
We apply the cobordism formula \eqref{eq:CobordismFormulaIntForm} to the \spinu structure
$\ft_n$ with
\begin{equation}
\label{eq:vanishing_cobordism_theorem_proof_delta_choice}
\delta:=\frac{1}{2}\dim \bL^{\asd}_{\ft_n}=c(X)+4\chi_h(X)-3n-1
\quad\text{and}\quad
\eta_c:=0,
\end{equation}
and
claim that Proposition \ref{prop:VanishingASDPairing} implies that the
left-hand-side of \eqref{eq:CobordismFormulaIntForm} vanishes.
The assumption \eqref{eq:VanishingSumAssump1} and the identity \eqref{eq:LevelOfSWInLa=0} imply that for $c_1(\fs)\in B(X)$,
\begin{equation}
\label{eq:La=0,ellge0}
2\ell(\ft_n,\fs)=4\chi_h(X)-2n \ge 0.
\end{equation}
Assumption \eqref{eq:VanishingSumAssump3},
the definition of $\delta$, and
\eqref{eq:La=0,ellge0} imply that
\begin{equation}
\label{eq:delta,ell,A}
\delta -2m\ge \delta-2\ell-2m= c(X)-n-1-2m \ge 0.
\end{equation}
Thus, $\delta-2m\ge 0$ and so condition \eqref{eq:VanishingASDCondition1}
of Proposition \ref{prop:VanishingASDPairing} holds.

The choice of $\delta$ and $\eta_c$ imply that $\delta+\eta_c=\frac{1}{2}\dim \bL^{\asd}_{\ft_n}$,
so condition \eqref{eq:VanishingASDCondition2}
of Proposition \ref{prop:VanishingASDPairing} holds.

Assumption \eqref{eq:VanishingSumAssump2} implies that
$-1>-n$, so $-3n-1>-4n$.  This inequality, our choice of $\delta$,
and \eqref{eq:ASDDimenForLa=0}
imply that
$$
\delta=c(X)+4\chi_h(X)-3n-1>c(X)+4\chi_h(X)-4n=
\frac{1}{2}\dim M^w_\ka,
$$
so condition \eqref{eq:VanishingASDCondition3}
of Proposition \ref{prop:VanishingASDPairing} holds.
Thus, all three conditions of Proposition \ref{prop:VanishingASDPairing}
hold and the
left-hand-side of \eqref{eq:CobordismFormulaIntForm} vanishes when applied with
the given values of $\delta$, $\eta_c$, and the \spinu structure $\ft_n$.
Under these conditions, equation \eqref{eq:CobordismFormulaIntForm} becomes
\begin{equation}
\label{eq:CobordismFormulaIntForm1}
0
=
-\sum_{\fs\in\overline{\Red}(\ft_n)}
(-1)^{\frac{1}{2}(w^2-\si)+\frac{1}{2}(w^2+w\cdot c_1(\fs))}
\#\left(\bar\sV(h^{\delta-2m}x^m)\cap \bL_{\ft_n,\fs}\right).
\end{equation}
For each $\fs\in\overline{\Red}(\ft_n)$,
equation \eqref{eq:MainEquation} in Theorem \ref{thm:SWLinkPairing}
implies that each term  in the sum on the right-hand-side of
\eqref{eq:CobordismFormulaIntForm1}
contains a factor of $\SW_X(\fs)$.
The terms in this sum given by
$\fs\in\overline{\Red}(\ft_n)$ with $c_1(\fs)\notin B(X)$ then vanish.
Hence, the sum in \eqref{eq:CobordismFormulaIntForm1} over
$\overline{\Red}(\ft_n)$
can be written as a double sum, over $K\in B(X)$ and then over $\fs\in c_1^{-1}(K)$
\begin{equation}
\label{eq:CobordismFormulaIntForm2}
\begin{aligned}
0{}&=
\sum_{\fs\in\overline{\Red}(\ft_n)}
(-1)^{\frac{1}{2}(w^2-\si)+\frac{1}{2}(w^2+w\cdot c_1(\fs))}
\#\left(\bar\sV(h^{\delta-2m}x^m)\cap \bL_{\ft_n,\fs}\right)
\\
{}&=
\sum_{\fs\in c_1^{-1}(K)}
\sum_{K\in B(X)}
(-1)^{\frac{1}{2}(w^2-\si)+\frac{1}{2}(w^2+w\cdot c_1(\fs))}
\#\left(\bar\sV(h^{\delta-2m}x^m)\cap \bL_{\ft_n,\fs}\right).
\end{aligned}
\end{equation}
Because we have assumed that $w\in H^2(X;\ZZ)$ is characteristic, we have
$w^2\equiv \si(X)\pmod 8$ by \cite[Lemma 1.2.20]{GompfStipsicz}, so
$$
(-1)^{\frac{1}{2}(w^2-\si(X))}=1.
$$
Our assumption that $\La=c_1(\ft_n)=0$
from \eqref{eq:CharClassOfSpinuWithLa0}
implies that all the terms in equation \eqref{eq:MainEquation}
with a factor of $\langle \La,h\rangle^j$ with $j>0$ vanish.
Thus, applying equation \eqref{eq:MainEquation} in
Theorem \ref{thm:SWLinkPairing}
to the terms in \eqref{eq:CobordismFormulaIntForm2}
and noting that $\ell = 2\chi_h(X)-n$ by \eqref{eq:LevelOfSWInLa=0}
 yields
\begin{equation}
\label{eq:CobordismFormulaIntForm3}
\begin{aligned}
0{}&=
\sum_{\fs\in c_1^{-1}(K)}\sum_{K\in B(X)}
(-1)^{\frac{1}{2}(w^2+w\cdot c_1(\fs))}
\#\left(\bar\sV(h^{\delta-2m}x^m)\cap \bL_{\ft_n,\fs}\right)
\\
{}&=
\sum_{\fs\in c_1^{-1}(K)}
\sum_{K\in B(X)}
(-1)^{\frac{1}{2}(w^2+w\cdot K)}
SW_X(\fs)
\\
{}&\quad\times
\sum_{\begin{subarray}{l}i+2k\\=\delta-2m\end{subarray}}
a_{i,0,k}(\chi_h(X),c_1^2(X),0,0,m,2\chi_h(X)-n)
\langle K,h\rangle^i
Q_X(h)^k.
\end{aligned}
\end{equation}
By the definition of $\SW_X'(K)$ in \eqref{eq:DefineCohomSW}, we can
rewrite \eqref{eq:CobordismFormulaIntForm3} as
\begin{equation}
\label{eq:CobordismFormulaIntForm4}
\begin{aligned}
0{}&=
\sum_{K\in B(X)}
(-1)^{\frac{1}{2}(w^2+w\cdot K)}
SW_X'(K)
\\
{}&\quad\times
\sum_{\begin{subarray}{l}i+2k\\=\delta-2m\end{subarray}}
a_{i,0,k}(\chi_h(X),c_1^2(X),0,0,m,2\chi_h(X)-n)
\langle K,h\rangle^i
Q_X(h)^k.
\end{aligned}
\end{equation}
Because the coefficient $a_{i,0,k}(\chi_h(X),c_1^2(X),0,0,m,2\chi_h(X)-n)$
does not depend on $K\in B(X)$,
we can use the abbreviation $a_{i,0,k}$ in \eqref{eq:CoeffAbbreviation}
to rewrite \eqref{eq:CobordismFormulaIntForm4} as
\begin{equation}
\label{eq:CobordismFormulaIntForm5}
\begin{aligned}
0{}&=
\sum_{\begin{subarray}{l}i+2k\\=\delta-2m\end{subarray}}
a_{i,0,k} Q_X(h)^k
\sum_{K\in B(X)}
(-1)^{\frac{1}{2}(w^2+w\cdot K)}
SW_X'(K)\langle K,h\rangle^i
\\
{}&=
\sum_{\begin{subarray}{l}i+2k\\=\delta-2m\end{subarray}}
a_{i,0,k}
\SW_X^{w,i}(h)
Q_X(h)^k
\quad\text{(by \eqref{eq:ReduceSWPolyToK}).}
\end{aligned}
\end{equation}
Because the coefficients $a_{i,0,k}$ in
\eqref{eq:CobordismFormulaIntForm5} vanish for $k>\ell=2\chi_h(X)-n$
by Theorem \ref{thm:SWLinkPairing},
we can rewrite \eqref{eq:CobordismFormulaIntForm5} as
\begin{equation}
\label{eq:CobordismFormulaIntForm6}
0=
\sum_{k=0}^{\ell}
a_{\delta-2m-2\ell+2k,0,\ell-k}
\SW_X^{w,\delta-2m-2\ell+2k}(h)
Q_X(h)^{\ell-k}.
\end{equation}
From \eqref{eq:delta,ell,A} and the definition  $A=c(X)-n-2m-1$
in the statement of the theorem, we have
$\delta-2m-2\ell+2k=A+2k$.  Substituting that
equality and $\ell=2\chi_h(X)-n$ into \eqref{eq:CobordismFormulaIntForm6}
completes the proof.
\end{proof}

\begin{rmk}
\label{rmk:VaryingCoeff}
As discussed in the beginning of this section, we work with
a \spinu structure $\ft$ with $c_1(\ft)=0$ in order to ensure that
the coefficients $a_{i,j,k}$ appearing in \eqref{eq:MainEquation}
do not depend on $K\in B(X)$.
Thus, after reversing the order of summation in
\eqref{eq:CobordismFormulaIntForm4} we can pull these coefficients
out in front of the inner sum over $K\in B(X)$ to
get  the
expression \eqref{eq:CobordismFormulaIntForm5}
involving the Seiberg--Witten polynomials.
Hence, the choice of \spinu structure with $c_1(\ft)=0$
is a necessary step in the argument.
\end{rmk}

\section{The leading term computation}
\label{sec:LeadingOrderComp}
To show that
equation \eqref{eq:VanishingCobordismReducedForm} is non-trivial, we
now demonstrate, in a computation similar to the proof of  \cite[Theorem 6.1.1]{KotschickMorgan},
that the coefficient of the term in \eqref{eq:VanishingCobordismReducedForm}
including the highest power of $Q_X$ is
non-zero.

\begin{prop}
\label{prop:LeadingTerm}
Continue the notation and assumptions of Theorem \ref{thm:VanishingCobordism}.
In addition, assume that there is a class $K\in B(X)$ with $K\neq 0$.
Let $m$ and $n$ be non-negative integers satisfying the conditions
\eqref{eq:VanishingSumAssump}.  Define
$A:=c(X)-n-2m-1$, and $\delta:=c(X)+4\chi_h(X)-3n-1$, and $\ell=2\chi_h(X)-n$.
Then
\begin{equation}
\label{eq:HighestPowerOfIntersectionForm}
a_{A,0,\ell}(\chi_h(X),c_1^2(X),0, 0,m,\ell)
=
(-1)^{m+\ell}
2^{\ell-\delta}
\displaystyle{\frac{(\delta-2m)!}{\ell! A!}} .
\end{equation}
\end{prop}

\begin{rmk}
Although the computation of
the precise value of the coefficient in \eqref{eq:HighestPowerOfIntersectionForm}
is quite delicate,
we are fortunate to require only the result that
$a_{A,0,\ell}$
is non-zero.
\end{rmk}

\begin{rmk}
The methods in this
section allow one to compute the coefficients $a_{i,j,\ell}$
in greater generality (for example, without the assumption that $c_1(\ft)=0$).
Because Theorem \ref{thm:SCST} does not require greater generality and indeed, as noted
in Remark \ref{rmk:VaryingCoeff}, requires the assumption that $c_1(\ft)=0$, we omit the
proof of the more general result in the interest of clarity.
\end{rmk}

\subsection{A neighborhood of a top stratum point}
Let $\bx\in\Sym^\ell(X)$
be in the top stratum of $\Sym^\ell(X)$.  In this section,
we collect some results needed in the proof of
Proposition \ref{prop:LeadingTerm} about the topology of a
neighborhood of $M_\fs\times\{\bx\}$  in $\bar\sM_\ft/S^1$.
We describe this
using the language of virtual neighborhoods (also called Kuranishi models).
(See McDuff and Wehrheim \cite{McDuff_Wehrheim_2017plms, McDuff_Wehrheim_2017gt, McDuff_Wehrheim_2018} for one approach to their detailed definition and for references therein to other approaches.) For additional details that we omit here for the sake of brevity, see Feehan and Leness \cite{FL5}.
Recall that a \emph{smoothly-stratified space} is a topological space stratified by smooth manifolds \cite[Section 21.1]{KMBook}.
For the purposes of this
article,
a \emph{virtual neighborhood} of a subspace $Y$ of a
smoothly-stratified
space $Z$
is a
smoothly-stratified
space $W$ with $Y\subset W$,
 together with a
smoothly-stratified
homeomorphism between $Z$ and
and a subspace $Z_W\subset W$ whose intersection with each stratum $W_\Si$ of $W$ is given by
the zero locus of a
transversely-vanishing section of a vector bundle $V_\Si\to W_\Si$,
referred to as the \emph{obstruction bundle},
with the property that
the restriction of this homeomorphism to $Y\subset Z$ is the identity
map
on $Y$.
Although the definition of the intersection numbers appearing in
equation \eqref{eq:CobordismFormulaIntForm} (and thus defining the
coefficient \eqref{eq:HighestPowerOfIntersectionForm}) requires
a smooth structure on the link $\bL_{\ft,\fs}$, the equality \eqref{eq:LinkPairingAsCohomology}
turns these intersection numbers into a cohomological pairing which allows us
to work in the topological category.
To keep the exposition simple, we shall leave discussions of smooth structures
to the proof   of \eqref{eq:LinkPairingAsCohomology} in \cite{FL5} as much as possible.

\subsubsection{The lower-level moduli space}
Because $\dim M_\fs=0$, the virtual normal bundle construction
of \cite[Theorem 3.21]{FL2a} gives a homeomorphism
between
a neighborhood in $\sM_{\ft(\ell)}/S^1$ of a point in $M_\fs$
and $\bchi_{\ft(\ell),\fs}^{-1}(0)/S^1$, where
\begin{equation}
\bchi_{\ft(\ell),\fs}:
\CC^{r_N}\to\CC^{r_\Xi}
\end{equation}
is a continuous, $S^1$-equivariant map with $0\in g_{\ft(\ell),\fs}^{-1}(0)$
and which is smooth away from the origin and vanishes transversely away from the origin.
The dimensions satisfy
\begin{equation}
\label{eq:DimensionRelations}
r_N-r_\Xi
=
\frac{1}{2}\dim\sM_{\ft(\ell)}
=
\frac{1}{2}\dim\sM_\ft -3\ell .
\end{equation}
We further note that because $\dim M_\fs=0$ and $M_\fs$ is compact and oriented,
$M_\fs$ is a finite
set of points.  If $1\in H^0(M_\fs;\ZZ)$ is
a generator given by an orientation of $M_\fs$, then
\begin{equation}
\label{eq:ZeroDimSW}
\langle
1,[M_\fs]
\rangle
=
\SW_X(\fs),
\end{equation}
as this pairing is just the count with sign of the points in the oriented moduli space $M_\fs$.

\subsubsection{The neighborhood of $M_\fs\times\{\bx\}$}
In
\cite[Chapters 6 \& 7]{FL5}, we constructed a virtual neighborhood $\bar\sM^{\vir}_{\ft,\fs}$
of $M_\fs\times\Sym^\ell(X)$ in $\bar\sM_\ft$
that admits a continuous, surjective map (see
\cite[Lemma 6.9.2]{FL5}),
\begin{equation}
\label{eq:ProjToSymmProduct}
\pi_X: \bar\sM^{\vir}_{\ft,\fs}
\to\Sym^\ell(X).
\end{equation}
The space $\bar\sM^{\vir}_{\ft,\fs}$ is stratified by smooth manifolds and contains a
subspace $\bar\sO_{\ft,\fs}$ which is homeomorphic to
a neighborhood $\bar\sU_{\ft,\fs}$ of $M_\fs\times\Sym^\ell(X)$ in $\bar\sM_\ft$.
Let $\sO_{\ft,\fs}$ be
the intersection of $\bar\sO_{\ft,\fs}$ with the top stratum of $\bar\sM^{\vir}_{\ft,\fs}$.
Then $\sO_{\ft,\fs}$ is the zero locus of a transversely vanishing
section of a vector bundle of rank $2r_\Xi+2\ell$ over the top stratum of $\bar\sM^{\vir}_{\ft,\fs}$
and $\sO_{\ft,\fs}$ is diffeomorphic to the top stratum of $\bar\sU_{\ft,\fs}$.
The top stratum
of $\bar\sM^{\vir}_{\ft,\fs}$ has dimension
determined by
$$
\frac{1}{2} \dim \sM^{\vir}_{\ft,\fs}=\frac{1}{2} \dim\sM_\ft + r_\Xi+\ell.
$$
There is an $S^1$ action on $\bar\sM^{\vir}_{\ft,\fs}$ which restricts to the
$S^1$ action on $\bar\sM_\ft$
discussed in Section \ref{subsubsec:Compactification}.
This $S^1$ action is free on the complement of its fixed point
set, $M_\fs\times\Sym^\ell(X)\subset \bar\sM^{\vir}_{\ft,\fs}$.

Let $\Delta\subset\Sym^\ell(X)$ be the `big diagonal', given by points
$\{x_1,\dots,x_\ell\}$ where $x_i=x_j$ for some $i\neq j$.
For $\bx\in  \Sym^\ell(X)\setminus\Delta$,
let $U$ be an open set,
$$
\bx \in U
\Subset
\Sym^\ell(X)\setminus\Delta,
$$
and let $\tilde U\subset X^\ell$ be the pre-image of $U$ under the
branched cover $X^\ell\to \Sym^\ell(X)$.
Let $\CSO(3)$ be the open cone on $\SO(3)$.
For $U$ sufficiently small, we define
\begin{equation}
\label{eq:GenericPointNgh}
N(\ft,\fs,U)
:=
M_\fs\times
\CC^{r_N}\times \CSO(3)^\ell \times_{\fS_\ell} \tilde U,
\end{equation}
where $\fS_\ell$ is the symmetric group on $\ell$ elements, acting diagonally by permutation
on the $\ell$ factors in $\CSO(3)^\ell$ and $\tilde U$.
Because $U$ is contained in
the top stratum of $\Sym^\ell(X)$,
the construction of $\bar\sM^{\vir}_{\ft,\fs}$ in
\cite[Section 6.6]{FL5}
and the map $\pi_X$ in
\cite[Lemma 6.9.2]{FL5} imply that
there is a commutative diagram,
\begin{equation}
\label{eq:GenericPointNghProj}
\begin{CD}
N(\ft,\fs,U)
@>\bga>>
\bar\sM^{\vir}_{\ft,\fs}
\\
@VVV @V \pi_X VV
\\
U @>>> \Sym^\ell(X)
\end{CD}
\end{equation}
where
\begin{enumerate}
\item
The horizontal maps are open embeddings,
\item
The vertical map on the left is projection onto the factor $U$,
\item
The image of $\bga$
is a neighborhood of $M_\fs\times\{\bx\}$ in $\bar\sM^{\vir}_{\ft,\fs}$, and
\item
The embedding $\bga$ is equivariant with respect to the diagonal $S^1$ action on
the factors of $\CC$ and $\SO(3)$ in
\eqref{eq:GenericPointNgh} and the $S^1$ action on $\bar\sM^{\vir}_{\ft,\fs}$.
\end{enumerate}
Observe that because $U$ is in the top stratum of $\Sym^\ell(X)$, the group
$\fS_\ell$ acts freely on $\tilde U$.  Hence, for $\bx\in U$, the pre-image of $\bx$
under the left vertical arrow in the diagram \eqref{eq:GenericPointNghProj}
is
$$
M_\fs\times
\CC^{r_N}\times \CSO(3)^\ell\times\{\bx\}.
$$
The commutativity of the diagram
\eqref{eq:GenericPointNghProj} implies that for
$\bx\in U$, the embedding $\bga$ defines
a homeomorphism,
\begin{equation}
\label{eq:PreImageOfGenericPoint}
M_\fs\times
\CC^{r_N}\times \CSO(3)^\ell\times\{\bx\}
\to
\pi_X^{-1}(\bx).
\end{equation}
Note that for $\bx \in U$ represented by $\{x_1,\ldots,x_\ell\}$, each $x_l$ has multiplicity one, by definition of $\Delta$, for $1\leq l \leq \ell$.

\begin{rmk}
The virtual neighborhood $\bar\sM^{\vir}_{\ft,\fs}$ is a union of
the domains of the gluing maps defined in \cite{FL3}.  Therefore,
the space \eqref{eq:GenericPointNgh} can be understood as follows.
The factor $M_\fs\times \CC^{r_N}$ represents pairs of `almost monopoles'
on the \spinu structure $\ft(\ell)$.  The factors $\CSO(3)$ represent
centered, charge-one, framed instantons
on $S^4$ which are spliced onto
pairs defined by $M_\fs\times\CC^{r_n}$ at the points $\{x_1,\dots,x_\ell\} \subset X$
defined by the factor $U \Subset \Sym^\ell(X) \less \Delta$.
This gluing construction is described in detail in \cite{FL3}, \cite{FLLevelOne},
\cite{FL5} and is similar to that described for the moduli space of anti-self-dual
connections in \cite[Section 7.2]{DK}, \cite[Section 3.4]{FrM}, and \cite{TauSelfDual, TauIndef, TauFrame}.
\end{rmk}

For the cone point $c\in\CSO(3)$,
define $\bc_\ell \in \CSO(3)^\ell$ by
$$
\bc_\ell= \underbrace{\{c\}\times\{c\}\times\dots\times\{c\}}_{\text{$\ell$ copies}} \in \CSO(3)^\ell.
$$
Because $\bc_\ell$ is a fixed point of the $\fS_\ell$ action on $\CSO(3)^\ell$,
\begin{equation}
\label{eq:ReducibleStratum}
\bga^{-1}\left(M_\fs\times\Sym^\ell(X)\right)
=
M_\fs\times \{0\}\times \{\bc_\ell\}\times U
\subset
M_\fs\times\CC^{r_N}\times \CSO(3)^\ell\times_{\fS_\ell} \tilde U,
\end{equation}
where $\bga$ is the embedding in \eqref{eq:GenericPointNghProj}.

\subsubsection{The link and its branched cover}

In
\cite[Section 8.1.1]{FL5}, we constructed a link
$\bL^{\vir}_{\ft,\fs}\subset\bar\sM^{\vir}_{\ft,\fs}/S^1$
of $M_\fs\times\Sym^\ell(X)\subset \bar\sM^{\vir}_{\ft,\fs}/S^1$.
We will need the following description of $\pi_X^{-1}(\bx)\cap \bL^{\vir}_{\ft,\fs}$
and a branched cover of this space.

\begin{lem}
\label{lem:LinkOfGenericPoint}
For $\bx\in\Sym^\ell(X)\setminus\Delta$,
the space
$\pi_X^{-1}(\bx)\cap\bL^{\mathrm{vir}}_{\ft,\fs}$
is
homeomorphic to the link of
$$
M_\fs\times
\{0\}\times \{\bc_\ell\}\times\{\bx\}
\quad\text{in}\quad
M_\fs
\times
\CC^{r_N}\times_{S^1}
\CSO(3)^\ell\times\{\bx\}.
$$
\end{lem}

\begin{proof}
From the description in \eqref{eq:ReducibleStratum} of the
intersection of the image $\bga$
with $M_\fs\times\Sym^\ell$ and by the $S^1$ equivariance
of this embedding,
we see that pre-image of $M_\fs\times\Sym^\ell(X)$
under the homeomorphism \eqref{eq:PreImageOfGenericPoint}
is
$$
M_\fs\times
\{0\}\times \{\bc_\ell\}\times\{\bx\}.
$$
Because the homeomorphism \eqref{eq:PreImageOfGenericPoint} is $S^1$ equivariant,
it identifies the link of the preceding space in the $S^1$ quotient
of the domain of \eqref{eq:PreImageOfGenericPoint}
with
$\pi_X^{-1}(\bx)\cap\bL^{\vir}_{\ft,\fs}$, as asserted.
\end{proof}

The computations in our proof of Proposition
\ref{prop:LeadingTerm} require the following branched cover of this link.

\begin{lem}
\label{lem:BranchedCoverOfLink}
There is a degree
$(-1)^{\ell} 2^{r_N+\ell-1}$ branched cover
$$
\tilde f:M_\fs\times\CC\PP^{r_N+2\ell-1} \to \pi_X^{-1}(\bx)\cap\bL^{\mathrm{vir}}_{\ft,\fs}.
$$
If $\nu$ is the first Chern class of the $S^1$ action on $\bar\sM^{\mathrm{vir}}_{\ft,\fs}$,
then
\begin{equation}
\label{eq:PullbackOfS1Action}
\tilde f^*\nu=1\times 2 \tilde \nu,
\end{equation}
where $1\in H^0(M_\fs;\ZZ)$  satisfies \eqref{eq:ZeroDimSW}
and $\tilde\nu\in H^2(\CC\PP^{r_N+2\ell-1};\ZZ)$ satisfies
\begin{equation}
\label{eq:NegHyperPlaneFormula}
\langle \tilde\nu^{r_N+2\ell-1},[\CC\PP^{r_N+2\ell-1}]\rangle=(-1)^{r_N+2\ell-1}.
\end{equation}
\end{lem}

\begin{proof}
The product of the degree $(-2)$ branched cover (see \cite[p. 423]{OzsvathBlowup} for an explanation
of the sign)
$\CC^2\to \CSO(3)$ with the map $z\to z^2$ on the factors of $\CC$
gives a degree $(-1)^\ell 2^{r_N+\ell}$ branched cover,
$$
M_\fs\times\CC^{r_N+2\ell}
\to
M_\fs\times
\CC^{r_N}\times \CSO(3)^\ell,
$$
mapping $M_\fs\times\{0\}$ to $M_\fs\times\{0\}\times\{\bc_\ell\}$
and
which is $S^1$ equivariant if $S^1$ acts with weight two on the image.
Consequently, this map takes the link of $M_\fs\times\{0\}$ in its domain
to the link of $M_\fs\times\{0\}\times\{\bc_\ell\}$ in its image.
By Lemma \ref{lem:LinkOfGenericPoint}, the link of $M_\fs\times\{0\}\times\{\bc_\ell\}$
in the $S^1$ quotient of
$M_\fs\times\CC^{r_N}\times \CSO(3)^\ell$ is homeomorphic to
$\pi_X^{-1}(\bx)\cap \bL^{\vir}_{\ft,\fs}$.
Hence, there is a degree
$(-1)^{\ell} 2^{r_N+\ell-1}$ branched cover
$$
\tilde f:M_\fs\times\CC\PP^{r_N+2\ell-1} \to \pi_X^{-1}(\bx)\cap \bL^{\vir}_{\ft,\fs}.
$$
Because this map doubles the weight of the $S^1$ action, $\tilde f^*\nu$ is twice the
first Chern class of the $S^1$ bundle,
$$
M_\fs\times\left( \CC^{r_n+2\ell}\setminus\{0\}\right)\times_{S^1}\CC
\to
M_\fs\times\CC\PP^{r_N+2\ell-1},
$$
whose first Chern class is $1\times\tilde\nu$, where $\tilde\nu$ is the negative of
the hyperplane class.
\end{proof}

\subsection{Multilinear algebra}
\label{subsec:MultLineAlg}
The proof of Proposition \ref{prop:LeadingTerm} requires us to consider the count of the number points in the intersection with
$\bL_{\ft,\fs}$ in \eqref{eq:MainEquation} as a symmetric multilinear map on $H_2(X;\RR)$ rather than a polynomial.
We thus introduce some notation to translate between the two concepts.

For a finite-dimensional, real vector space $V$, let $S_d(V)$ be the vector space
of $d$-linear, symmetric maps $M:V^{\otimes d}\to\RR$, and let $P_d(V)$ be the vector space
of degree $d$ homogeneous polynomials on $V$.  The map $\Phi:S_d(V)\to P_d(V)$ defined
by $\Phi(M)(v)=M(v,\dots,v)$ is an isomorphism of vector spaces (see \cite[Section 6.1.1]{FrM}).
For $M_i\in S^{d_i}(V)$, we define a product on $S_\bullet(V)=\oplus_{d\ge 0}S_d(V)$ by
\begin{equation}
\label{eq:SymmAlgProduct}
\begin{aligned}
{}&(M_1M_2)(h_1,\dots,h_{d_1+d_2})
\\
{}&:=
\frac{1}{(d_1+d_2)!}\sum_{\si\in\fS_{d_1+d_2}}
M_1(h_{\si(1)},\dots,h_{\si(d_1)})
M_2(h_{\si(d_1+1)},\dots,h_{\si(d_1+d_2)}),
\end{aligned}
\end{equation}
where $\fS_d$ is the symmetric group on $d$ elements.
When $S_\bullet(V)$ has this product and $P_\bullet(V)$ has
its usual product, $\Phi$ is an algebra isomorphism.

\begin{lem}
\label{lem:PolarizOfSWLinkPairing}
Continue the assumptions and notation of Proposition \ref{prop:LeadingTerm}.
For $n\in\NN$ as in Proposition \ref{prop:LeadingTerm},
let $\ft_n$ be a \spinu structure satisfying \eqref{eq:CharClassOfSpinuWithLa0}.
Then
\begin{equation}
\begin{aligned}
\label{eq:MultilinearVersion}
{}&
\#\left(\bar\sV(h_1\cdots h_{\delta-2m} x^m)\cap \bL_{\ft_n,\fs}\right)
\\
{}&
\quad=
\frac{SW_X(\fs)}{(\delta-2m)!}
\sum_{\begin{subarray}{l}i+2k\\=\delta-2m\end{subarray}}
\sum_{\si\in\fS_{\delta-2m}} a_{i,0,k}(\chi_h(X),c_1^2(X),0,0,m,\ell)
\\
{}&\qquad\times
\prod_{u=1}^i\langle K,h_{\si(u)}\rangle
\prod_{u=1}^{(\delta-2m-i)/2}Q_X(h_{\si(i+2u-1)},h_{\si(i+2u)}),
\end{aligned}
\end{equation}
where
$\fS_{\delta-2m}$ is the symmetric group on $(\delta-2m)$ elements and
$K=c_1(\fs)$.
\end{lem}

\begin{proof}
Because we are assuming $\La=0$, all terms on the right-hand-side
of \eqref{eq:MainEquation} containing a factor of $\langle \La,h\rangle^j$
with $j>0$ vanish.
Applying $\Phi$ to both sides of \eqref{eq:MultilinearVersion} then
yields \eqref{eq:MainEquation} in
Theorem \ref{thm:SWLinkPairing}.  Because $\Phi$ is an isomorphism,
the result follows.
\end{proof}

The following corollary shows that the computation which will
yield the coefficient appearing in \eqref{eq:HighestPowerOfIntersectionForm}.

\begin{cor}
\label{cor:LeadingTermPolarizedOfSWLinkPairing}
Continue the notation and hypotheses of Lemma \ref{lem:PolarizOfSWLinkPairing}
and abbreviate $A=\delta-2m-2\ell$.
There is a class $h\in \Ker K\subset H_2(X;\RR)$ with $Q_X(h)=1$ and if
\begin{equation}
\label{eq:HomologyInKerK}
h_u=h \in \Ker K\subset H_2(X;\RR)
\quad \text{for $A+1\le u\le \delta-2m$,}
\end{equation}
then
\begin{equation}
\label{eq:MultilinearVersion1}
\begin{aligned}
{}&
\#\left(\bar\sV(h_1\cdots h_{\delta-2m} x^m)\cap \bL_{\ft,\fs}\right)
\\
{}&
\quad=
\frac{SW_X(\fs)A!(2\ell)!}{(\delta-2m)!}
a_{A,0,\ell}(\chi_h(X),c_1^2(X),0,0,m,\ell)
\prod_{u=1}^{A}\langle K,h_{u}\rangle.
\end{aligned}
\end{equation}
\end{cor}

\begin{proof}
Because $b^+(X)\ge 3$, $Q_X$ is positive on a three-dimensional subspace
$P\subset H_2(X;\RR)$.
The codimension of $\Ker K\subset H_2(X;\RR)$ is at most one so $P\cap \Ker K$
has dimension at least two.  Hence, there is
a class $h'\in \Ker K$ with $Q_X(h')>0$.  Then $h=h'/Q_X(h')^{1/2}\in \Ker K$ satisfies $Q_X(h)=1$, as required.

The assumption \eqref{eq:HomologyInKerK} implies that
only $A$ elements of $\{h_1,\dots,h_{\delta-2m}\}$
are not in
$\Ker K$.
Therefore,
in all terms in the sum in \eqref{eq:MultilinearVersion}
with $k<\ell$, we have $i=\delta-2m-2k>\delta-2m-2\ell=A$ and thus in such a term,
the product of the $i>A$ factors,
$$
\prod_{u=1}^i\langle K,h_{\si(u)}\rangle,
$$
must vanish.
Hence, all terms with $k<\ell$ in the sum in \eqref{eq:MultilinearVersion}
vanish.  We know the terms with $k>\ell$ vanish because the coefficients
$a_{i,0,k}$ with $k>\ell$ vanish by Theorem \ref{thm:SWLinkPairing},
so only the terms with $k=\ell$ are non-zero.
Thus, \eqref{eq:MultilinearVersion} and the equality $(\delta-2m-A)/2=\ell$ imply that
\begin{equation}
\label{eq:MultilinearVersion1.1}
\begin{aligned}
{}&
\#\left(\bar\sV(h_1\cdots h_{\delta-2m} x^m)\cap \bar\sW^{\eta_c}\cap \bL_{\ft,\fs}\right)
\\
{}&
\quad=
\frac{SW_X(\fs)}{(\delta-2m)!}
\sum_{\si\in\fS_{\delta-2m}} a_{A,0,\ell}(\chi_h(X),c_1^2(X),0,0,m,\ell)
\\
{}&\qquad\times
\prod_{u=1}^A\langle K,h_{\si(u)}\rangle
\prod_{u=1}^{\ell}Q_X(h_{\si(A+2u-1)},h_{\si(A+2u)}).
\end{aligned}
\end{equation}
If $\si_0\in\fS_{\delta-2m}$ and 
$\si_0(u) > A$ for $u\le A=i$, then the term given by that $\si_0$ in the sum on the right-hand side of
\eqref{eq:MultilinearVersion1.1} contains the factor
$$
\prod_{u=1}^A\langle K,h_{\si_0(u)}\rangle
\prod_{u=1}^{\ell}Q_X(h_{\si_0(A+2u-1)},h_{\si_0(A+2u)}),
$$
which vanishes by the assumption \eqref{eq:HomologyInKerK} that
$h_u\in\Ker K $ for $u>A$.
Thus, for $h_u$ as given in \eqref{eq:HomologyInKerK}, the sum over
$\fS_{\delta-2m}$ in \eqref{eq:MultilinearVersion1.1}
reduces to a sum over the subset
$$
\fS_{\delta-2m}(A)
:=
\{\si\in\fS_{\delta-2m}: \text{$\si(u)\le A$ for all $u\le A$}\}.
$$
The pigeonhole principle then implies that
elements of $\fS_{\delta-2m}(A)$ preserve
the subsets $\{1,\dots,A\}$ and $\{A+1,\dots,\delta-2m\}$.
If we identify $\fS_A$ and $\fS_{2\ell}$ (using $2\ell=\delta-2m-A$)
with the subgroups of $\fS_{\delta-2m}$ which are the identity on
$\{A+1,\dots,\delta-2m\}$ and $\{1,\dots,A\}$ respectively,
then there is an isomorphism
\begin{equation}
\label{eq:SymmGrpIsom}
S:
\fS_{A}\times \fS_{2\ell}
\to
\fS_{\delta-2m}(A), \quad S(\si_1,\si_2)=\si_1\si_2.
\end{equation}
This isomorphism, the identity \eqref{eq:MultilinearVersion1.1}, and the equality $(\delta-2m-A)/2=\ell$ imply that
\begin{equation}
\begin{aligned}
\label{eq:MultilinearVersion2}
{}&
\#\left(\bar\sV(h_1\cdots h_{\delta-2m} x^m)\cap \bar\sW^{\eta_c}\cap \bL_{\ft,\fs}\right)
\\
{}&
\quad=
\frac{SW_X(\fs)}{(\delta-2m)!}
\sum_{\si_1\in\fS_{A}}
\sum_{\si_2\in\fS_{2\ell}}
a_{A,0,\ell}(\chi_h(X),c_1^2(X),0,0,m,\ell)
\\
{}&\qquad\times
\prod_{u=1}^A\langle K,h_{\si_1(u)}\rangle
\prod_{u=1}^{\ell}Q_X(h_{\si_2(A+2u-1)},h_{\si_2(A+2u)}).
\end{aligned}
\end{equation}
Observe that for all $\si_1\in\fS_{A}$,
$$
\prod_{u=1}^A\langle K,h_{\si_1(u)}\rangle
=
\prod_{u=1}^A\langle K,h_u\rangle
$$
while for all $\si_2\in \fS_{2\ell}$,
$$
\prod_{u=1}^{\ell}Q_X(h_{\si_2(A+2u-1)},h_{\si_2(A+2u)})
=
Q_X(h)^\ell=1.
$$
Thus, all the $|\fS_{A}||\fS_{2\ell}|=A!(2\ell)!$   terms in the double sum in \eqref{eq:MultilinearVersion2}
are equal and we can rewrite \eqref{eq:MultilinearVersion2} as
\begin{equation}
\begin{aligned}
\label{eq:MultilinearVersion3}
{}&
\#\left(\bar\sV(h_1\cdots h_{\delta-2m} x^m)\cap \bar\sW^{\eta_c}\cap \bL_{\ft,\fs}\right)
\\
{}&
\quad=
\frac{SW_X(\fs)A!(2\ell)!}{(\delta-2m)!}
a_{A,0,\ell}(\chi_h(X),c_1^2(X),0,0,m,\ell)
\prod_{u=1}^A\langle K,h_{u}\rangle.
\end{aligned}
\end{equation}
This completes the proof of the corollary.
\end{proof}

\subsection{Cohomology classes and duality}
By
\cite[Proposition 9.1.1]{FL5}, there are a topological space
$\bL^{\vir}_{\ft,\fs}\subset\bar\sM^{\vir}_{\ft,\fs}/S^1$
with fundamental class $[\bL^{\vir}_{\ft,\fs}]$
and cohomology classes
$$
\barmu_p(h_i), \ \bar\mu_c, \ \bar e_I, \ \bar e_s
\in
H^\bullet\left(\bar\sM^{\vir}_{\ft,\fs}/S^1\setminus
\left(M_\fs\times\Sym^\ell(X)\right);\RR\right),
$$
such that
\begin{equation}
\label{eq:LinkPairingAsCohomology}
\begin{aligned}
{}&\#\left(
\bar\sV(h_1\cdots h_{\delta-2m} x^m)
\cap \bL_{\ft,\fs}
\right)
\\
{}&\quad=
\langle
\barmu_p(h_1)\smile\dots\smile\barmu_p(h_{\delta-2m})\smile\barmu_p(x)^m
\smile\bar e_I\smile\bar e_s,
[\bL^{\vir}_{\ft,\fs}]
\rangle.
\end{aligned}
\end{equation}
The cohomology classes $\bar e_I$ and $\bar e_s$ are extensions over the Uhlenbeck
compactification of the Euler class of components of the obstruction bundle.
The cohomology classes $\barmu_p(h_i)$ and $\barmu_p(x)$
are extensions of the classes $\mu_p(h_i)$ and $\mu_p(x)$
defined in Section \ref{subsubsec:CohomClass}.
For $\beta\in H_\bullet(X;\RR)$, the cohomology class $S^\ell(\beta)\in H^{4-\bullet}(\Sym^\ell(X);\RR)$
is defined by the property that, for $\tilde\pi:X^\ell\to\Sym^\ell(X)$ denoting
the degree-$\ell!$ branched covering map,
$$
\tilde\pi^*S^\ell(\beta)=\sum_{i=1}^n \pi_i^*\PD[\beta],
$$
where $\pi_i:X^\ell \to X$ is projection onto the $i$-th factor.
Thus (compare \cite[p. 454]{KotschickMorgan}),
\begin{equation}
\label{eq:SymmClassPairing}
\begin{aligned}
{}&\langle
S^\ell(h_1)\smile \dots\smile S^\ell(h_{2\ell+k}),[\Sym^\ell(X)]\rangle
\\
{}&\quad=
\begin{cases}
\displaystyle\frac{(2\ell)!}{\ell! 2^\ell} Q_X^\ell(h_1,\dots,h_{2\ell}) \PD[\bx] & \text{if $k=0$,}
\\
0 & \text{if $k>0$,}
\end{cases}
\end{aligned}
\end{equation}
where  $\bx\in \Sym^\ell(X)\setminus\Delta$ is a point in the top stratum.
Note that if $h_u=h$ for $u=1,\dots,2\ell$, then by the definition of the product in
\eqref{eq:SymmAlgProduct},
\begin{equation}
\label{eq:PolarizOfQell}
Q_X^\ell(h_1,\dots,h_{2\ell})
=
Q_X(h)^\ell.
\end{equation}
From
\cite[Proposition 9.1.1]{FL5}
and \cite[Definitions 9.4.8 \& 9.5.8]{FL5}
we have,
denoting $K=c_1(\fs)$, $\La=c_1(\ft)$, $h\in H_2(X;\RR)$, and a generator $x\in H_0(X;\ZZ)$,
\begin{equation}
\label{eq:CohomClasses}
\begin{aligned}
\barmu_p(h) &=\frac{1}{2}\langle \La-K,h\rangle \nu +\pi_X^*S^\ell(h),
\\
\barmu_p(x) &=-\frac{1}{4}  \nu^2 +\pi_X^*S^\ell(x),
\\
\bar e_\fs &= (-\nu)^{r_\Xi},
\end{aligned}
\end{equation}
where $\nu$ is the first Chern class of the $S^1$ action
on $\bar\sM^{\vir}_{\ft,\fs}$
and $\pi_X$ is defined in \eqref{eq:ProjToSymmProduct}.

\subsection{The computation}
We can now give the

\begin{proof}[Proof of Proposition \ref{prop:LeadingTerm}]
For $n\in\NN$,  as appearing in Proposition \ref{prop:LeadingTerm},
let $\ft_n$ be a \spinu structure satisfying \eqref{eq:CharClassOfSpinuWithLa0}.
We will apply Corollary \ref{cor:LeadingTermPolarizedOfSWLinkPairing}
to verify the expression \eqref{eq:HighestPowerOfIntersectionForm} for the coefficient $a_{A,0,\ell}$.
From the definitions of $\delta$, $A$, and $\ell$ in the statement of Proposition \ref{prop:LeadingTerm}
and the expression for $\delta$ in \eqref{eq:vanishing_cobordism_theorem_proof_delta_choice},
\begin{equation}
\label{eq:IndexRelations}
A+2\ell+2m=\delta
=\frac{1}{2}\dim\bL^{\asd}_{\ft_n}
=\frac{1}{2}\dim\sM_{\ft_n}-1.
\end{equation}
By hypothesis in Proposition \ref{prop:LeadingTerm}, there is a class $K\in B(X)$ with $K\neq 0$.  Let $\fs\in\Spinc(X)$
satisfy $c_1(\fs)=K$.
As in the proof of Corollary \ref{cor:LeadingTermPolarizedOfSWLinkPairing},
there are classes $h_0,h_0'\in H_2(X;\RR)$ which satisfy
\begin{equation}
\label{eq:AssumptionsOnh}
\langle K,h_0\rangle= 0, \quad Q_X(h_0)=1, \quad \langle K,h_0'\rangle=-1.
\end{equation}
Define
\begin{equation}
\label{eq:AssumptionsOnhu}
h_u :=
\begin{cases}
h_0' & \text{for $1\le u\le A$,}
 \\
h_0 & \text{for $A+1\le u\le \delta-2m$.}
 \end{cases}
\end{equation}
Corollary \ref{cor:LeadingTermPolarizedOfSWLinkPairing},
the identity
\eqref{eq:AssumptionsOnh}, and the definition \eqref{eq:AssumptionsOnhu} imply that
\begin{equation}
\begin{aligned}
\label{eq:MultilinearVersion1a}
{}&
\#\left(\bar\sV(h_1\cdots h_{\delta-2m} x^m)\cap \bL_{\ft_n,\fs}\right)
\\
{}&
\quad=
(-1)^A
\frac{SW_X(\fs)A!(2\ell)!}{(\delta-2m)!}
a_{A,0,\ell}(\chi_h(X),c_1^2(X),0,0,m,\ell).
\end{aligned}
\end{equation}
We now use the results of the previous sections to compute
the left-hand-side of \eqref{eq:MultilinearVersion1a}.
Applying \eqref{eq:LinkPairingAsCohomology} with $\ft = \ft_n$ gives
\begin{equation}
\label{eq:LinkPairingAsCohomology1}
\begin{aligned}
{}&\#\left(
\bar\sV(h_1\cdots h_{\delta-2m} x^m)
\cap \bL_{\ft_n,\fs}
\right)
\\
{}&\quad=
\langle
\barmu_p(h_1)\smile\cdots\smile\barmu_p(h_{\delta-2m})\smile\barmu_p(x)^m
\smile\bar e_I\smile\bar e_s,
[\bL^{\vir}_{\ft_n,\fs}]
\rangle.
\end{aligned}
\end{equation}
By \eqref{eq:CohomClasses}
(with $\Lambda = c_1(\ft_n) = 0$),
\eqref{eq:AssumptionsOnh}, and \eqref{eq:AssumptionsOnhu},
$$
\barmu_p(h_u)=
\begin{cases}
\frac{1}{2}\nu+\pi_X^*S^\ell(h_0') & \text{for $1\le u\le A$}
\\
\pi_X^*S^\ell(h_0) & \text{for $A+1\le u\le \delta-2m$}.
\end{cases}
$$
Substituting the preceding expressions for $\barmu_p(h_u)$ and the expressions for $\barmu_p(x)$ and $\bar e_s$ from \eqref{eq:CohomClasses}
into \eqref{eq:LinkPairingAsCohomology1}
and using the equality $\delta-2m-A=2\ell$ in \eqref{eq:IndexRelations}
gives
\begin{equation}
\label{eq:LinkPairingAsCohomology2}
\begin{aligned}
{}&\#\left(
\bar\sV(h_1\cdots h_{\delta-2m} x^m)
\cap \bL_{\ft_n,\fs}
\right)
\\
{}&\quad=
\left\langle
\left(
\frac{1}{2}\nu+\pi_X^*S^\ell(h_0')
\right)^A
\smile
\left(
\pi_X^*S^\ell(h_0)
\right)^{2\ell}
\smile
\left( -\frac{1}{4} \nu^2+\pi_X^* S^\ell(x)\right)^m \right.
\\
{}&\qquad\qquad
\left.\smile
(-\nu)^{r_\Xi}
\smile \bar e_I,
[\bL^{\vir}_{\ft_n,\fs}]
\right\rangle.
\end{aligned}
\end{equation}
Applying
the computations \eqref{eq:SymmClassPairing} and \eqref{eq:PolarizOfQell} and our assumption
in \eqref{eq:AssumptionsOnh} that $Q_X(h_0)=1$ to \eqref{eq:LinkPairingAsCohomology2}
yields
\begin{equation}
\label{eq:LinkPairingAsCohomology3}
\begin{aligned}
{}&\#\left(
\bar\sV(h_1\cdots h_{\delta-2m} x^m)
\cap \bL_{\ft_n,\fs}
\right)
\\
{}&\quad=
\frac{(2\ell)!}{\ell! 2^\ell}
\left
\langle
\left(
\frac{1}{2}\nu+\pi_X^*S^\ell(h_0')
\right)^A
\smile
\left( -\frac{1}{4} \nu^2+\pi_X^* S^\ell(x)\right)^m
\right.
\\
{}&\qquad\qquad\qquad
\left.
\smile
(-\nu)^{r_\Xi}
\smile
\bar e_I,
\pi_X^*\PD[\bx]\cap [\bL^{\vir}_{\ft_n,\fs}]
\right\rangle.
\end{aligned}
\end{equation}
Because
$$
S^\ell(h_0')\smile\PD[\bx]
=
0
=
S^\ell(x)\smile\PD[\bx],
$$
by dimension-counting on $\Sym^\ell(X)$, the identity
\eqref{eq:LinkPairingAsCohomology3} simplifies to
\begin{equation}
\label{eq:LinkPairingAsCohomology4}
\begin{aligned}
{}&\#\left(
\bar\sV(h_1\cdots h_{\delta-2m} x^m)
\cap \bL_{\ft_n,\fs}
\right)
\\
{}&\quad=
\frac{(2\ell)!}{\ell! 2^\ell}
(-1)^{m+r_\Xi}
\,2^{-A-2m}
\langle
\nu^{A+2m+r_\Xi}\smile\bar e_I,
[\pi_X^{-1}(\bx)\cap \bL^{\vir}_{\ft_n,\fs}]
\rangle.
\end{aligned}
\end{equation}
Finally, we apply the computation from \cite[Lemma 4.12]{FLLevelOne},
where it is proved that
the restriction of $\bar e_I$ to
$\pi_X^{-1}(\bx)\cap \bL^{\vir}_{\ft_n,\fs}$ equals
$(-2)^{-\ell}\nu^\ell$, to rewrite \eqref{eq:LinkPairingAsCohomology4} as
\begin{equation}
\label{eq:LinkPairingAsCohomology5}
\begin{aligned}
{}&\#\left(
\bar\sV(h_1\cdots h_{\delta-2m} x^m)
\cap \bL_{\ft_n,\fs}
\right)
\\
{}&\quad =
\frac{(2\ell)!}{\ell! 2^\ell}
(-1)^{m+r_\Xi+\ell}
\,2^{-A-2m-\ell}
\langle
\nu^{A+2m+r_\Xi+\ell},
[\pi_X^{-1}(\bx)\cap \bL^{\vir}_{\ft_n,\fs}]
\rangle.
\end{aligned}
\end{equation}
Now observe that, because \eqref{eq:DimensionRelations} gives
$$
r_N-r_\Xi=\frac{1}{2}\dim\sM_{\ft_n} - 3\ell = \delta +1 - 3\ell,
$$
the equality \eqref{eq:IndexRelations} implies that
we have
\begin{equation}
\label{eq:IndexDimCounting}
A+2m+r_\Xi+\ell
=
\delta-\ell+r_\Xi
=r_N+2\ell-1.
\end{equation}
Hence, using the branched cover $\tilde f$ of degree
$(-1)^{\ell} 2^{r_N+\ell-1}$ in
Lemma \ref{lem:BranchedCoverOfLink}, we can write
\begin{align*}
{}&\langle
\nu^{A+2m+r_\Xi+\ell},
[\pi_X^{-1}(\bx)\cap \bL^{\vir}_{\ft_n,\fs}]
\rangle
\\
{}&\quad=
(-1)^\ell
2^{-r_N-\ell+1}
\langle
\nu^{r_N+2\ell-1},
\tilde f_*[M_\fs\times\CC\PP^{r_N+2\ell-1}]
\rangle
\\
{}&\quad=
(-2)^\ell
\langle
(1\times\tilde\nu)^{r_N+2\ell-1},[M_\fs\times\CC\PP^{r_N+2\ell-1}]
\rangle
\quad\text{(by \eqref{eq:PullbackOfS1Action})}
\\
{}&\quad=
(-2)^\ell
\langle 1,[M_\fs]\rangle
\times
\langle \tilde\nu^{r_N+2\ell-1},\CC\PP^{r_N+2\ell-1}\rangle,
\end{align*}
where the final identity follows from \cite[Theorem 5.6.13]{Spanier}.
Thus, applying \eqref{eq:ZeroDimSW} and \eqref{eq:NegHyperPlaneFormula} to the preceding expression yields
\begin{equation}
\label{eq:PairingPullbackToBranchCover}
\langle
\nu^{A+2m+r_\Xi+\ell},
[\pi_X^{-1}(\bx)\cap \bL^{\vir}_{\ft_n,\fs}]
\rangle
=
(-1)^{r_N+1+\ell}2^\ell \SW_X(\fs).
\end{equation}
Combining \eqref{eq:LinkPairingAsCohomology5} and \eqref{eq:PairingPullbackToBranchCover}
implies that under the assumptions \eqref{eq:AssumptionsOnh}
on $h_u$,
\begin{equation}
\label{eq:LinkPairing6}
\begin{aligned}
{}&\#\left(
\bar\sV(h_1\cdots h_{\delta-2m} x^m)
\cap \bL_{\ft_n,\fs}
\right)
\\
{}&\quad =
\frac{(2\ell)!}{\ell! 2^\ell}
(-1)^{m+r_\Xi+r_N+1}
2^{-A-2m}\SW_X(\fs).
\end{aligned}
\end{equation}
Comparing \eqref{eq:MultilinearVersion1a} and \eqref{eq:LinkPairing6}
gives
\begin{align*}
{}&
(-1)^A
\frac{SW_X(\fs)A!(2\ell)!}{(\delta-2m)!}
a_{A,0,\ell}(\chi_h(X),c_1^2(X),0,0,m,\ell)
\\
{}&\quad =
\frac{(2\ell)!}{\ell! 2^\ell}
(-1)^{m+r_\Xi+r_N+1}
2^{-A-2m}\SW_X(\fs),
\end{align*}
which we solve to get
\begin{equation}
\label{eq:SolvingLeadingCoeff}
a_{A,0,\ell}(\chi_h(X),c_1^2(X),0,0,m,\ell)
=
\frac{(\delta-2m)!}{\ell!A! }
(-1)^{A+m+r_\Xi+r_N+1}
2^{-A-2m-\ell}.
\end{equation}
Equation \eqref{eq:IndexRelations} implies that $-A-2m-\ell=\ell-\delta$,
while
\eqref{eq:IndexDimCounting} implies that
$$
A+m+r_\Xi+r_N+1
\equiv
\ell+m\pmod 2.
$$
Hence, \eqref{eq:SolvingLeadingCoeff} yields the
desired equality \eqref{eq:HighestPowerOfIntersectionForm} and this completes the proof of Proposition \ref{prop:LeadingTerm}.
\end{proof}

\section{Vanishing coefficients}
\label{sec:CoeffOfVanishing}
We now determine the coefficients $a_{i,0,k}$ with $i\ge c(X)-3$ appearing in
\eqref{eq:VanishingCobordismReducedForm}.  Although, as pointed out in Remark \ref{rmk:AdditionalParameterInCoeff},
the coefficients in \eqref{eq:VanishingCobordismReducedForm} are not those determined in
\cite[Proposition 4.8]{FL6}, the techniques used in the proof of \cite[Proposition 4.8]{FL6} also determine
the coefficients $a_{i,0,k}$ with $i\ge c(X)-3$ appearing in
\eqref{eq:VanishingCobordismReducedForm}.

\begin{prop}
\label{prop:DeterminingCoeff}
Continue the hypothesis and notation of Theorem \ref{thm:VanishingCobordism} and assume that
\begin{equation}
\label{eq:ParityAssumption}
n\equiv 1\pmod 2.
\end{equation}
Then for $p\ge c(X)-3$ and
an integer $k\geq 0$ such that
$p+2k=c(X)+4\chi_h(X)-3n-1-2m$,
$$
a_{p,0,k}(\chi_h(X),c_1^2(X),0,0,m,2\chi_h-n)=0.
$$
\end{prop}

We prove Proposition \ref{prop:DeterminingCoeff} by showing that
on certain standard four-manifolds,
the vanishing result
\eqref{eq:VanishingCobordismReducedForm}
forces each of the coefficients in the sum to be zero
by using the following generalization of \cite[Lemma VI.2.4]{FrM}.

\begin{lem}
\label{lem:AlgCoeff}
\cite[Lemma 4.1]{FL6}
Let $V$ be a finite-dimensional real vector space.
Let $T_1,\dots,T_n $ be linearly independent elements of the dual
space $V^*$.
Let $Q$ be a quadratic form on $V$ which is non-zero
on $\cap_{i=1}^n\Ker T_i$.  Then $T_1,\dots,T_n,Q$ are algebraically
independent in the sense that if
$F(z_0,\dots,z_n)\in \RR[z_0,\dots,z_n]$
and $F(Q,T_1,\dots,T_n):V\to\RR$ is the zero map, then $F(z_0,\dots,z_n)$
is the zero element of $\RR[z_0,\dots,z_n]$.
\end{lem}

In \cite[Section 4.2]{FL6}, we used the manifolds constructed by
Fintushel, Park and Stern  in \cite{FSParkSympOneBasic}
to give the following family of standard
four-manifolds.

\begin{lem}
\label{lem:ExampleManifolds}
For every integer $q\ge 2$, there is a standard four-manifold
$X_q$ of Seiberg--Witten simple type satisfying
\begin{subequations}
\begin{align}
\label{eq:CharNumOfExamples}
{}& \chi_h(X_q)=q \quad\hbox{and}\quad c(X_q)=3,
\\
\label{eq:BasicClassesOfExamples}
{}& B(X_q)=\{\pm K\} \quad\hbox{and}\quad K\neq 0,
\\
\label{eq:IntFormNonZeroInExamples}
{}& \text{The restriction of $Q_{X_q}$ to $\Ker K\subset H_2(X_q;\RR)$ is non-zero.}
\end{align}
\end{subequations}
\end{lem}

We write the blow-up of $X_q$ at $r$ points as $X_q(r)$, so
\begin{align*}
\chi_h(X_q(r)) &= \chi_h(X_q)=q,
\\
c_1^2(X_q(r)) &= c_1^2(X_q)-r,
\\
c(X_q(r)) &= c(X_q)+r=r+3,
\end{align*}
where we recall from \eqref{eq:CharNumbers} that $c(X) := \chi_h(X)-c_1^2(X)$.
We consider both the homology and cohomology of $X_q$ as subspaces
of the homology and cohomology of $X_q(r)$, respectively.
Let $e_u^*\in H^2(X_q(r);\ZZ)$ be the Poincar\'e dual of the $u$-th exceptional class.
Let $\pi_u:(\ZZ/2\ZZ)^r\to\ZZ/2\ZZ$ be projection onto the $u$-th factor.
For $\varphi\in (\ZZ/2\ZZ)^r$ and $K\in B(X_q)$, we define
\begin{equation}
\label{eq:Definitions_Kvarphi_and_K0}
K_\varphi := K+\sum_{u=1}^r (-1)^{\pi_u(\varphi)}e_u^*
\quad\hbox{and}\quad
K_0 := K+\sum_{u=1}^r e_u^*.
\end{equation}
Then,
by the blow-up formula for Seiberg--Witten invariants
\cite[Theorem 14.1.1]{Froyshov_2008},
\begin{equation}
\label{eq:BlowUpBasicClasses}
\begin{aligned}
B'(X_q(r))
&=
\{ K_\varphi:\varphi\in (\ZZ/2\ZZ)^r\},
\\
\SW_{X_q(r)}(K_\varphi) &= \SW_{X_q}(K).
\end{aligned}
\end{equation}
In preparation for our application of Lemma \ref{lem:AlgCoeff}, we have the

\begin{lem}
\label{lem:AlgIndSetOnBlowUp}
Let $q\ge 2$ and $r\ge 0$ be  integers.
Let $X_q(r)$ be the blow-up of the four-manifold $X_q$ given in Lemma \ref{lem:ExampleManifolds}
at $r$ points.  Then the set
$$
\{K,e_1^*,\dots,e_r^*,Q_{X_q(r)}\}
$$
is algebraically independent in the sense of Lemma \ref{lem:AlgCoeff}
for the vector space $H_2(X_q(r);\RR)$.
\end{lem}

\begin{proof}
The cohomology classes $K,e_1^*,\dots,e_r^*$ are linearly independent
in $H^2(X_q(r);\RR)$.
The restriction of $Q_{X_q(r)}$ to the intersection of the kernel of
these cohomology classes equals the restriction of $Q_{X_q}$ to the kernel of $K$ in $H_2(X_q;\RR)$,
which is non-zero by \eqref{eq:IntFormNonZeroInExamples}.  Hence,
Lemma \ref{lem:AlgCoeff} implies that $\{K,e_1^*,\dots,e_r^*,Q_{X_q(r)}\}$ is
algebraically independent.
\end{proof}

\begin{proof}[Proof of Proposition \ref{prop:DeterminingCoeff}]
Because $c(X)\ge 3$, if $q=\chi_h(X)$ and $r=c(X)-3\ge 0$, then
\begin{equation}
\label{eq:chih_and_c1squared_for_X_and_Xq(r)_equal}
\chi_h(X)=\chi_h(X_q(r))
\quad\hbox{and}\quad
c_1^2(X)=c_1^2(X_q(r))
\end{equation}
by Lemma \ref{lem:ExampleManifolds} and so
\begin{equation}
\label{eq:a_i0k_coefficients_for_X_and_Xq(r)_equal}
a_{i,0,k}(\chi_h(X),c_1^2(X),0,0,m,\ell)
=
a_{i,0,k}(\chi_h(X_q(r)),c_1^2(X_q(r)),0,0,m,\ell).
\end{equation}
As in the proof of Theorem \ref{thm:VanishingCobordism},
the assumptions on $m$ and $n$ allow us to apply
the cobordism formula \eqref{eq:CobordismFormulaIntForm}
with a \spinu structure $\ft_n$ on $X_q(r)$
satisfying \eqref{eq:CharClassOfSpinuWithLa0},
$\tilde w\in H^2(X_q(r);\ZZ)$ characteristic,
\begin{align*}
\delta &:= c(X)+4\chi_h(X)-3n-1
\quad\hbox{(from \eqref{eq:vanishing_cobordism_theorem_proof_delta_choice})}
\\
&\,= c(X_q(r))+4\chi_h(X_q(r))-3n-1
\quad\hbox{(by \eqref{eq:chih_and_c1squared_for_X_and_Xq(r)_equal}),}
\end{align*}
and $\ell(\ft_n,\fs)=2\chi_h(X_q(r))-n$ from \eqref{eq:LevelOfSWInLa=0}
to get (see \eqref{eq:CobordismFormulaIntForm4})
\begin{equation}
\label{eq:VanishingSumB1}
\begin{aligned}
0{}&=
\sum_{K\in B(X_q(r))}
(-1)^{\frac{1}{2}(\tilde w^2+\tilde w\cdot K)}SW'_{X_q(r)}(K)
\\
{}&\qquad\times
\sum_{\begin{subarray}{l}i+2k\\=\delta-2m\end{subarray}}
a_{i,0,k}(\chi_h(X_q(r)),c_1^2(X_q(r)),0,0,m,\ell)
\langle K,h\rangle^i
Q_{X_q(r)}(h)^k.
\end{aligned}
\end{equation}
Because $\SW_X(-K)=(-1)^{\chi_h(X)}\SW_X(K)$ by \cite[Corollary 6.8.4]{MorganSWNotes},
the set $B(X_q(r))$ is closed under the action of $\{\pm 1\}$.
Let $B'(X_q(r))$ be a fundamental domain for
the action of $\{\pm 1\}$ on
$B(X_q(r))$.  We will rewrite \eqref{eq:VanishingSumB1}
as a sum over
$B'(X_q(r))$ by combining the terms given by $K$ and $-K$.  First observe that
\begin{align*}
\frac{1}{2}(\tilde w^2+\tilde w\cdot(- K))
{}&\equiv
\frac{1}{2}(\tilde w^2+\tilde w\cdot K) + \tilde w\cdot K \pmod 2
\\
{}&\equiv
\frac{1}{2}(\tilde w^2+\tilde w\cdot K)+ K^2\pmod 2\quad\text{(because $\tilde w$ is characteristic)}
\\
{}&\equiv
\frac{1}{2}(\tilde w^2+\tilde w\cdot K)+ c_1^2(X)\pmod 2\quad\text{(by  \eqref{eq:SWST})}.
\end{align*}
Combining this equality with $\SW_X(-K)=(-1)^{\chi_h(X)}\SW_X(K)$ yields
\begin{align*}
{}&(-1)^{\frac{1}{2}(\tilde w^2+\tilde w\cdot K)}SW'_{X_q(r)}(K)\langle K,h\rangle^i
+
(-1)^{\frac{1}{2}(\tilde w^2+\tilde w\cdot(- K))}SW'_{X_q(r)}(-K)\langle -K,h\rangle^i
\\
{}&\quad=
(-1)^{\frac{1}{2}(\tilde w^2+\tilde w\cdot K)}SW'_{X_q(r)}(K)\langle K,h\rangle^i \left(1+(-1)^{c_1^2(X)+\chi_h(X)+i}\right)
\\
{}&\quad=
(-1)^{\frac{1}{2}(\tilde w^2+\tilde w\cdot K)}SW'_{X_q(r)}(K)\langle K,h\rangle^i
\left(1+(-1)^{c(X)+i}\right).
\end{align*}
Because $n\equiv 1\pmod 2$ by our
assumption \eqref{eq:ParityAssumption}, we have $\delta=c(X)+4\chi_h(X)-3n-1\equiv c(X)\pmod 2$,
so $\delta=i+2k\equiv c(X)\pmod 2$ implies that $c(X)+i\equiv 0\pmod 2$. Hence,
the preceding identity simplifies to give
\begin{equation}
\label{eq:CombiningTerms2}
\begin{aligned}
{}&(-1)^{\frac{1}{2}(\tilde w^2+\tilde w\cdot K)}SW'_{X_q(r)}(K)\langle K,h\rangle^i
+
(-1)^{\frac{1}{2}(\tilde w^2+\tilde w\cdot(- K))}SW'_{X_q(r)}(-K)\langle -K,h\rangle^i
\\
{}&\quad=
(-1)^{\frac{1}{2}(\tilde w^2+\tilde w\cdot K)}2\,SW'_{X_q(r)}(K)\langle K,h\rangle^i.
\end{aligned}
\end{equation}
Equation \eqref{eq:CombiningTerms2} allows us to rewrite \eqref{eq:VanishingSumB1} as
\begin{equation}
\label{eq:VanishingSumB'}
\begin{aligned}
0{}&=
\sum_{K\in B'(X_q(r))}
(-1)^{\frac{1}{2}(\tilde w^2+\tilde w\cdot K)}SW'_{X_q(r)}(K)
\\
{}&\qquad\times
\sum_{\begin{subarray}{l}i+2k\\=\delta-2m\end{subarray}}
2a_{i,0,k}(\chi_h(X_q(r)),c_1^2(X_q(r)),0,0,m,\ell)
\langle K,h\rangle^i
Q_{X_q(r)}(h)^k.
\end{aligned}
\end{equation}
If we abbreviate $a_{i,0,k}=a_{i,0,k}(\chi_h(X_q(r)),c_1^2(X_q(r)),0,0,m,\ell)$
and
$$
\eps(\tilde w,K_\varphi)=\frac{1}{2}(\tilde w^2+\tilde w\cdot K_\varphi),
$$
and use the description of $B'(X_q(r))$ in \eqref{eq:BlowUpBasicClasses}, then
\eqref{eq:VanishingSumB'} yields
\begin{equation}
\label{eq:VanishingSumB'2}
0=
\sum_{\varphi\in (\ZZ/2\ZZ)^r}
\sum_{\begin{subarray}{l}i+2k\\=\delta-2m\end{subarray}}
(-1)^{\eps(\tilde w,K_\varphi)}\SW'_{X_q(r)}(K_\varphi)
2a_{i,0,k}
\langle K_\varphi,h\rangle^i
Q_{X_q(r)}(h)^k.
\end{equation}
To apply Lemma \ref{lem:AlgCoeff} to \eqref{eq:VanishingSumB'2}
and get information about the coefficients $a_{i,0,k}$,
we will replace $B'(X_q(r))$ with the set $\{K,e_1^*,\dots,e_r^*\}$
appearing in Lemma \ref{lem:AlgIndSetOnBlowUp}.

Because $\tilde w\in H^2(X_q(r);\ZZ)$ is characteristic,
$\tilde w\cdot e_u^*\equiv (e_u^*)^2\equiv 1\pmod 2$.
Hence, the expression
$$
\eps(\tilde w,K_\varphi)
\equiv
\eps(\tilde w,K_0)
+\sum_{u=1}^r\frac{1}{2} \left((-1)^{\pi_u(\varphi)}-1\right)\tilde w\cdot e_u^* \pmod 2
$$
simplifies to give
\begin{equation}
\label{eq:OrientationChange}
\eps(\tilde w,K_\varphi)
\equiv
\eps(\tilde w,K_0)+\sum_{u=1}^r\pi_u(\varphi) \pmod 2.
\end{equation}
Using the definition \eqref{eq:Definitions_Kvarphi_and_K0} of $K_\varphi$,
we expand the factor $\langle K_\varphi,h\rangle^i$ in \eqref{eq:VanishingSumB'2} as
\begin{equation}
\label{eq:BinomExp}
\langle K_\varphi,h\rangle^i
=
\sum_{i_0+\cdots+i_r=i}
(-1)^{\sum_{u=1}^r \pi_u(\varphi)i_u}
\binom{i}{i_0 \cdots i_r}
\langle K,h\rangle^{i_0}
\prod_{u=1}^r \langle e_u^*,h\rangle^{i_u}.
\end{equation}
Substituting \eqref{eq:OrientationChange} and
\eqref{eq:BinomExp} into \eqref{eq:VanishingSumB'2},
yields
\begin{equation}
\label{eq:VanishingSumB'3}
\begin{aligned}
0&=
(-1)^{\eps(\tilde w,K_0)}\SW_{X_q}'(K)
\sum_{\varphi\in (\ZZ/2\ZZ)^r}
\sum_{\begin{subarray}{c}i_0+\cdots+i_r+2k\\=\delta-2m\end{subarray}}
\binom{i_0+\cdots+i_r}{i_0\cdots i_r}
(-1)^{\sum_{u=1}^r (1+i_u)\pi_u(\varphi)}
\\
{}&\qquad\times
2a_{i,0,k}
\langle K,h\rangle^{i_0}
\prod_{u=1}^r \langle e_u^*,h\rangle^{i_u}
Q_{X_q(r)}(h)^k.
\end{aligned}
\end{equation}
By Lemma \ref{lem:AlgIndSetOnBlowUp},
the set $\{K,e_1^*,\dots,e_r^*,Q_{X_q(r)}\}$
is algebraically independent and so the monomials
$$
K^{i_0} \left(\prod_{u=1}^r (e_u^*)^{i_u}\right) Q_X^k
$$
are linearly independent.
For the integer $p$ appearing in the statement of Proposition \ref{prop:DeterminingCoeff},
we have $p\ge c(X)-3$ by assumption, so $p\ge r$ by
the equality $r=c(X)-3$ preceding \eqref{eq:chih_and_c1squared_for_X_and_Xq(r)_equal}.
Hence, equation \eqref{eq:VanishingSumB'3}
and Lemma \ref{lem:AlgCoeff}
imply that the coefficient
of the term
$$
\langle K,h\rangle^{p-r}
\prod_{u=1}^r\langle e_u^*,h\rangle
Q_{X_q(r)}(h)^k
$$
in \eqref{eq:VanishingSumB'3}
must vanish.
Because $i_u=1$ for $u=1,\dots,r$ in this term and $p = i_0+\cdots+i_r$,
we can write this
coefficient as
\begin{align*}
{}&(-1)^{\eps(\tilde w,K_0)}\SW_{X_q}'(K)
\frac{p!}{(p-r)!}
\sum_{\varphi\in (\ZZ/2\ZZ)^r}
(-1)^{\sum_{u=1}^r 2\pi_u(\varphi)}
2a_{p,0,k}
\\
{}&\quad=
(-1)^{\eps(\tilde w,K_0)}\SW'_{X_q}(K)
\frac{p!}{(p-r)!}
\sum_{\varphi\in (\ZZ/2\ZZ)^r}
2a_{p,0,k}
\\
{}&\quad=
(-1)^{\eps(\tilde w,K_0)}\SW_{X_q}'(K)
\frac{p!2^{r+1}}{(p-r)!}a_{p,0,k}.
\end{align*}
Hence, the coefficient $a_{p,0,k}$ must vanish, as asserted,
and this concludes the proof of Proposition \ref{prop:DeterminingCoeff}.
\end{proof}

\section{Proof of the main result}
\label{sec:MainProof}

We will prove Theorem \ref{thm:SCST} by applying
the computations of the coefficients in
Proposition \ref{prop:LeadingTerm} and
Proposition \ref{prop:DeterminingCoeff} to the vanishing sum \eqref{eq:VanishingCobordismReducedForm}.

To apply Proposition \ref{prop:LeadingTerm}, we need
to assume that there is a class $K\in B(X)$ with $K\neq 0$.
We can make this assumption if we can replace $X$ with its blow-up
$\widetilde X$.
In the following lemma, we show that the superconformal simple type
condition is invariant under blow-up, allowing us to make the desired replacement
of $X$ with $\widetilde X$ in the proof of Theorem \ref{thm:SCST}.

\begin{lem}
\label{lem:SCSTBlowUp}
Let $X$ be a standard four-manifold of Seiberg--Witten simple type
with $c(X)\ge 3$.
Then $X$ has superconformal simple type if and only if its blow-up
$\widetilde X$ has superconformal simple type.
\end{lem}

\begin{proof}
If $X$ has superconformal simple type, then so does $\widetilde X$ by
\cite[Theorem 7.3.1]{MMPhep}.  We prove the converse.
If $c(X)\le 3$,
the result is trivial; we will show that if $c(X)\ge 4$ and $\widetilde X$
has superconformal simple type, then $X$ satisfies \eqref{eq:SWPolynomial}.
Note that $c(X)\ge 4$ implies that $c(\widetilde X)\ge 5$, so
$\widetilde X$ having superconformal simple type implies that
$\widetilde X$
satisfies \eqref{eq:SWPolynomial}.

Let $e^*\in H^2(\widetilde X;\ZZ)$ be the Poincar\'e dual
of the exceptional curve.
Let $w\in H^2(X;\ZZ)$  be characteristic, so $\tilde w:=w-e^*\in H^2(\widetilde X;\ZZ)$ is
also characteristic.
By \cite[Theorem 14.1.1]{Froyshov_2008},
$$
B(\widetilde X)=\{K\pm e^*: K\in B(X)\}
\quad\hbox{and}\quad
\SW_{\widetilde X}'(K\pm e^*)=\SW_X'(K).
$$
For $K\in B(X)$,
\begin{align*}
\frac{1}{2}(\tilde w^2+\tilde w\cdot(K+e))&\equiv \frac{1}{2} (w^2+w\cdot K) \pmod 2,
\\
\frac{1}{2}(\tilde w^2+\tilde w\cdot(K-e))&\equiv\frac{1}{2} (w^2+w\cdot K)+1\pmod 2.
\end{align*}
Then, by the expression for $\SW^{w,i}_X$ in \eqref{eq:ReduceSWPolyToK} and applying the sign identities just noted,
\begin{align*}
\SW^{\tilde w, i}_{\widetilde X}(h)
{}&=
\sum_{K\in B(X)}
(-1)^{\frac{1}{2}(\tilde w^2+\tilde w\cdot(K+e))}\SW_{\widetilde X}'(K+e^*)\langle K+e,h\rangle^i
\\
{}&\qquad+
(-1)^{\frac{1}{2}(\tilde w^2+\tilde w\cdot(K-e))}\SW_{\widetilde X}'(K-e^*)\langle K-e,h\rangle^i
\\
{}&=
\sum_{K\in B(X)} (-1)^{\frac{1}{2}(w^2+w\cdot K)}\SW_X'(K)
\sum_{u=0}^i \binom{i}{u} \left( 1 - (-1)^u\right)\langle K,h\rangle^{i-u}\langle e^*,h\rangle^u.
\end{align*}
Thus,
for all $h \in H_2(\widetilde X;\RR)$,
\begin{equation}
\label{eq:SWPolyBlowUp}
\SW^{\tilde w,i}_{\widetilde X}(h)
=
\sum_{u=0}^i \binom{i}{u} \SW^{w,i-u}_X(h) \left( 1 - (-1)^u\right)\langle e^*,h\rangle^u.
\end{equation}
As noted at the beginning of this proof, we can assume that $\widetilde X$ satisfies  \eqref{eq:SWPolynomial},
so
$\SW^{\tilde w,i}_{\widetilde X}$ vanishes for $i\le c(\widetilde X)-4$ by \eqref{eq:SWPolynomial}.
For $h_0\in H_2(X;\RR)$ and $e\in H_2(\widetilde X;\RR)$, the homology class of the exceptional
curve, and $s,t\in\RR$, equation \eqref{eq:SWPolyBlowUp} implies that
\begin{align*}
\SW^{\tilde w,i}_{\widetilde X}(s h_0+t e)
{}&=
\sum_{u=0}^i \binom{i}{u} \SW^{w,i-u}_X(sh_0) \left( 1 - (-1)^u\right)\langle e^*,te\rangle^u
\\
{}&=
\sum_{u=0}^i \binom{i}{u} \SW^{w,i-u}_X(h_0) \left( 1 - (-1)^u\right)(-1)^u s^{i-u}t^u.
\end{align*}
Assume $i\ge 1$.
Because $\SW^{\tilde w,i}_{\widetilde X}$ vanishes for $i\le c(\widetilde X)-4$, we have
\begin{align*}
0{}&=\left.\left(\frac{\rd^i}{\rd s^{i-1}\rd t}\SW^{\tilde w}_{\widetilde X}(s h_0+t e)\right)\right|_{s=t=0}
\\
{}&=-2 i\SW^{w,i-1}_X(h_0).
\end{align*}
Thus, $\SW^{w,i-1}_X$ vanishes for $1\le i\le c(\widetilde X)-4$ or $0\le i-1\le c(X)-4$, as required.
\end{proof}

Half of the polynomials $\SW^{w,i}_X$ vanish for the following trivial reasons.
(This result appears in  the remarks following \cite[Proposition 6.1.3]{MMPhep};
we include it here for completeness.)

\begin{lem}
\label{lem:SWPolyParity}
If $X$ is a standard four-manifold,
$i\ge 0$ is any integer obeying $c(X)+i\equiv 1\pmod 2$, and
$w\in H^2(X;\ZZ)$ is characteristic, then
$\SW_X^{w,i}$ vanishes.
\end{lem}

\begin{proof}
Because $\SW_X'(K)=(-1)^{\chi_h(X)}\SW_X'(-K)$ by \cite[Corollary 6.8.4]{MorganSWNotes},
the terms in \eqref{eq:DefineCohomSW}
corresponding to $K$ and $-K$ in $\SW^{w,i}_X$, namely
$$
(-1)^{\frac{1}{2} (w^2+w\cdot K)}\SW_X'(K)\langle K,h\rangle^i
\quad\text{and}\quad
(-1)^{\frac{1}{2} (w^2-w\cdot K)}\SW_X'(-K)\langle -K,h\rangle^i
$$
differ by the factor
$$
(-1)^{\chi_h(X) +w\cdot K+i}.
$$
Because $w$ and $K$ are characteristic,
we have
$w\cdot K\equiv K^2\equiv c_1^2(X)\pmod 2$.  Hence,
\begin{align*}
\chi_h(X)+w\cdot K + i
&\equiv
                         \chi_h(X)+K^2+i
                         \\
&\equiv
                              \chi_h(X)+c_1^2(X)+i
                              \\
&\equiv
c(X)+i\pmod 2.
\end{align*}
Thus, if $c(X)+i\equiv 1\pmod 2$, then the terms for $K$ and $-K$
in $\SW_X^{w,i}$ cancel and the function $\SW_X^{w,i}$ vanishes.
\end{proof}

The vanishing of the sum \eqref{eq:VanishingCobordismReducedForm}
will give information about
the Seiberg--Witten polynomial $\SW_X^{w,A}$ of degree $A=c(X)-n-2m-1$
which appears in this sum with a non-zero coefficient.
We write $A=c(X)-2v$, where
$v$ is a non-negative integer such that
$2v=n+2m+1$ as in the statement of Theorem \ref{thm:VanishingCobordism}, and compute some of the values
for this degree to which we can apply Theorem \ref{thm:VanishingCobordism}
and Proposition \ref{prop:DeterminingCoeff}.
Observe that if $n=3$, then the equality $2v=n+2m+1$ implies that $m=v-2$.

\begin{lem}
\label{lem:RequiredValues}
Let $X$ be a standard four-manifold with $c(X)\ge 3$.
For any $v\in\NN$
with $4\le 2v\le c(X)$,
the natural numbers $n=3$ and $m=v-2$
satisfy
the conditions
\eqref{eq:VanishingSumAssump}
in Theorem \ref{thm:VanishingCobordism} and
the parity condition \eqref{eq:ParityAssumption} in
Proposition \ref{prop:DeterminingCoeff}.
\end{lem}

\begin{proof}
Because $\chi_h(X)\ge 2$ for a standard four-manifold,
$n=3$  will
satisfy the conditions \eqref{eq:VanishingSumAssump1},
\eqref{eq:VanishingSumAssump2}, and \eqref{eq:ParityAssumption}.
The hypothesis $4\le 2v$ implies that $2m=2v-4\ge 0$ while
the hypothesis $2v\le c(X)$ implies that $2m=2v-4\le c(X)-4= c(X)-n-1$, which is the condition
\eqref{eq:VanishingSumAssump3}.
\end{proof}

\begin{rmk}
We note that the requirement $0<n-1$ in \eqref{eq:VanishingSumAssump2} implies that $n\ge 2$
and so $2v=n+1+2m\ge 3$.
Hence, the  methods of this article
do not imply that $\SW_X^{w,i}$
vanishes when $i > c(X) - 3$, which does not hold
in general.
\end{rmk}

\begin{proof}[Proof of Theorem \ref{thm:SCST}]
By Lemma \ref{lem:SCSTBlowUp}, it suffices to prove that the blow-up of $X$
has superconformal simple type.  Because $c_1^2(\widetilde X)=c_1^2(X)-1$,
we can assume $c_1^2(X)\neq 0$ by replacing $X$ with its blow up if necessary.
If we assume $c_1^2(X)\neq 0$ and $K\in B(X)$, then $K^2=c_1^2(X)\neq 0$ by \eqref{eq:SWST},
so $K\neq 0$.  Thus, we can assume $0\notin B(X)$
by replacing $X$ with its blow-up if needed.

We now abbreviate $c=c(X)$
and $\chi_h = \chi_h(X)$.
If $w\in H^2(X;\ZZ)$ is characteristic, then $\SW_X^{w,i}$ vanishes unless
$i\equiv c\pmod 2$  by Lemma \ref{lem:SWPolyParity}.  Thus,
it suffices to prove that $\SW_X^{w,c-2v}=0$ for $4\le 2v\le c$,
which we will do by induction on $v$.

By Lemma \ref{lem:RequiredValues}, the values $n=3$ and $m=v-2$
satisfy the conditions
\eqref{eq:VanishingSumAssump}
in Theorem \ref{thm:VanishingCobordism}.  Substituting these values
into \eqref{eq:VanishingCobordismReducedForm}
(noting that $A = c-n-2m-1 = c-2v$),
yields
\begin{equation}
\label{eq:VanishingCobordismReducedForm1}
0=
\sum_{k=0}^{2\chi_h-3}
a_{c-2v+2k,0,2\chi_h-3-k}
\SW_X^{w,c-2v+2k}(h)
Q_X(h)^{2\chi_h-3+k},
\end{equation}
where the coefficients $a_{i,0,k}$ are defined in \eqref{eq:CoeffAbbreviation}.
Because $n=3$ satisfies the assumption \eqref{eq:ParityAssumption},
Proposition \ref{prop:DeterminingCoeff} implies that
\begin{equation}
\label{eq:VanishingCoeffInCobordism}
a_{c-2v+2k,0,2\chi_h-3-k}=0
\quad
\text{for $2k-2v\ge -3$}.
\end{equation}
Because of our assumption that $0\notin B(X)$, an application of
Proposition \ref{prop:LeadingTerm} with $n=3$ gives
\begin{equation}
\label{eq:NonVanishingCoeffInCobordism}
a_{c-2v,0,2\chi_h-3}\neq 0.
\end{equation}
We now begin the induction on $v$.  If $2v=4$, the identity
\eqref{eq:VanishingCobordismReducedForm1} becomes
\begin{align*}
0{}&=\sum_{k=0}^{2\chi_h-3}
a_{c-4+2k,0,2\chi_h-3-k}
\SW^{w,c-4+2k}_X(h)Q_X(h)^{2\chi_h-3-k}
\\
&=
a_{c-4,0,2\chi_h-3}
\SW^{w,c-4}_X(h)Q_X(h)^{2\chi_h-3} \quad\text{(by \eqref{eq:VanishingCoeffInCobordism})},
\end{align*}
that is,
\begin{equation}
\label{eq:VanishingCobordismSumWithLa=0,v1}
0 = a_{c-4,0,2\chi_h-3} \SW^{w,c-4}_X(h)Q_X(h)^{2\chi_h-3}.
\end{equation}
Because $2v=4$, equations \eqref{eq:NonVanishingCoeffInCobordism} and \eqref{eq:VanishingCobordismSumWithLa=0,v1} imply that
$$
\SW^{w,c-4}_X(h)Q_X(h)^{2\chi_h-3}=0 \quad\text{for all $h\in H_2(X;\RR)$}.
$$
If $Z\subset H_2(X;\RR)$ is the (codimension-one) zero locus of $Q_X$, the preceding equality
implies that the polynomial $\SW^{w,c-4}_X$ vanishes
on the open, dense subset  $H_2(X;\RR)\setminus Z$ of $H_2(X;\RR)$ and hence
$\SW^{w,c-4}_X$ vanishes on $H_2(X;\RR)$, completing the proof of the initial case
of the induction on $v$.

For our induction hypothesis, we assume that $\SW^{w,c-2v'}_X=0$ for all $v'$ with $4\le 2v'<2v\le c$. We split the sum in \eqref{eq:VanishingCobordismReducedForm1} into three terms:
\begin{equation}
\label{eq:VanishingCobordismReducedForm2}
\begin{aligned}
0{}&=
a_{c-2v,0,2\chi_h-3}
\SW^{w,c-2v}_X(h)Q_X(h)^{2\chi_h-3}
\\
{}&\qquad+
\sum_{k=1}^{v-2}
a_{c-2v+2k,0,2\chi_h-3-k}
\SW^{w,c-2v+2k}_X(h)Q_X(h)^{2\chi_h-3-k}
\\
{}&\qquad
+
\sum_{k=v-1}^{2\chi_h-3}
a_{c-2v+2k,0,2\chi_h-3-k}
\SW^{w,c-2v+2k}_X(h)Q_X(h)^{2\chi_h-3-k}.
\end{aligned}
\end{equation}
If either of the two sums in \eqref{eq:VanishingCobordismReducedForm2} are sums over empty
indexing sets, then the notation is meant to indicate that those sums vanish.
We now show that the two sums will vanish even if their indexing sets are non-empty.
If we write $c-2v+2k=c-2(v-k)$ and define $v'=v-k$, then for $1\le k\le v-2$, we have
$v-1\ge v'\ge 2$.  Hence, by our induction hypothesis, we see that
\begin{equation}
\label{eq:VanishingByInduction}
\SW_X^{w,c-2v+2k}
=
\SW_X^{w,c-2v'}=0\quad\text{for $1\le k\le v-2$}.
\end{equation}
If $v-1\le k$, then $2k-2v\ge -3$ and so \eqref{eq:VanishingCoeffInCobordism} implies that
\begin{equation}
\label{eq:VanishingByCoeff}
a_{c-2v+2k,0,2\chi_h-3-k}=0\quad\text{for $v-1\le k\le 2\chi_h-3$}.
\end{equation}
The vanishing results \eqref{eq:VanishingByInduction} and \eqref{eq:VanishingByCoeff}
imply that the two sums in \eqref{eq:VanishingCobordismReducedForm2} vanish
while \eqref{eq:NonVanishingCoeffInCobordism} implies that the coefficient of
the first term on the right-hand-side of \eqref{eq:VanishingCobordismReducedForm2}
is non-zero.  Therefore, the identity \eqref{eq:VanishingCobordismReducedForm2} reduces to
\begin{equation}
\label{eq:VanishingCobordismReducedForm4}
0=
\SW^{w,c-2v}_X(h)Q_X(h)^{2\chi_h-3}.
 \end{equation}
If $Z$ is the zero locus of $Q_X$, then
\eqref{eq:VanishingCobordismReducedForm4} implies that
the polynomial $\SW^{w,c-2v}_X$ vanishes
on the open dense subset $H_2(X;\RR)\setminus Z$ of $H_2(X;\RR)$ and hence
$\SW^{w,c-2v}_X$ vanishes identically, completing the induction and the proof of Theorem \ref{thm:SCST}.
\end{proof}

\begin{proof}[Proof of Corollary \ref{cor:NumberOfBasicClasses}]
The lower bound on the number of basic classes on $X$ is true for manifolds of superconformal simple type by \cite[Theorem 8.1.1]{MMPhep}.  Hence the Corollary follows immediately from Theorem \ref{thm:SCST}.
\end{proof}

\begin{proof}[Proof of Corollary \ref{cor:WC}]
Witten's Conjecture follows from the $\SO(3)$-monopole cobordism formula,
\cite[Theorem 1]{FL5} for standard four-manifolds of superconformal simple type by \cite[Theorem 1]{FL7}.
Thus Corollary \ref{cor:WC} follows from \cite[Theorem 1]{FL7} and Theorem  \ref{thm:SCST}.
\end{proof}

\bibliography{/Users/pfeehan/Dropbox/LATEX/Bibinputs/master}
\bibliographystyle{amsplain-nodash}
\end{document}